\documentclass[titlepage,12pt]{article}
\usepackage{graphicx} 

\usepackage[utf8]{inputenc}
\usepackage{yfonts}
\usepackage{faktor}
\usepackage{amsfonts}
\usepackage{bbding}
\usepackage{bm}
\usepackage{fancyvrb}

\usepackage{tikz}
\usetikzlibrary{automata, positioning, arrows}

\usepackage[labelsep=space]{caption}

\usepackage{layout}
\usepackage{url}
\usepackage{proof}
\usepackage{mathrsfs}
\usepackage{amsmath,amsthm,amssymb,bussproofs}
\usepackage{braket}
\usepackage{graphicx}
\usepackage{mathtools}
\usepackage[all]{xy}
\DeclarePairedDelimiter{\ceil}{\lceil}{\rceil}

\makeatletter
\newcommand*\circled[2][1.6]{\tikz[baseline=(char.base)]{
    \node[shape=circle, draw, inner sep=1pt, 
        minimum height={\f@size*#1},] (char) {#2};}}
\makeatother

\newcommand{\deff}[1]{\textbf{\textit{#1}}}

\theoremstyle{definition}
\newtheorem{thm}{Theorem}[section]
\newtheorem{defn}[thm]{Definition}
\newtheorem{cor}[thm]{Corollary}
\newtheorem{prop}[thm]{Proposition}

\newtheorem{lemma}[thm]{Lemma}

\newtheorem{rmk}[thm]{Remark}

\newtheorem{eg}[thm]{Example}

\newcommand{\NN}{\mathbb{N}}

\DeclareMathOperator{\dom}{dom}

\title{Games with backtracking options corresponding to the ordinal analysis of $PA$}
\author{Eitetsu KEN}

\begin{document}

\maketitle

\begin{abstract}
We give another proof of ordinal analysis of $I\Sigma_{k}$-fragments of Peano Arithmetic which is free from cut-elimination of $\omega$-logic.
Our main tool is a direct witnessing argument utilizing game notion, motivated from the realm of proof complexity and bounded arithmetic.
\end{abstract}

\section{Introduction}
In \cite{FirstGentzenConsistencyProof}, Gentzen initiated so-called ordinal analysis of arithmetic.
In modern terminology, he determined proof-theoretical ordinals of $PA$.
The results include unprovability of transfinite induction up to $\epsilon_{0}$ in $PA$. 
(See also \cite{DetailedGentzenConsistencyProof} and \cite{TheCollectedPapersofGentzen}.)
One initial way to look at this was a corollary of Gentzen's consistency proof of $PA$ and G\"{o}del's second incompleteness theorem; $PA$ cannot prove transfinite induction up to $\epsilon_{0}$ since $PA$ could prove the consistency of itself otherwise.
In his latter paper \cite{DirectIndependenceofTI}, Gentzen gave a direct independence proof for it, a proof without using incompleteness theorem.
It was a modification of his consistency proof of $PA$; the argument known as cut-elimination of $\omega$-logic these days.
The ideas already sufficed to accomplish the ordinal analysis of $I\Sigma_{k}$-fragments of $PA$ introduced by \cite{fragmentsofarithmetic}.

On the other hand, game notion is a main tool in proof complexity.
Pebble game has been successful to characterize resolution width (\cite{ResolutionPebble}) and provability in $T^{1}_{2}(R)$ (\cite{reslogtonarrowresol}).
The technique can be regarded as a variation of the proof of completeness of (free-)cut-free sequent calculus, describing ways to climb up a given proof falsifying the formulae appearing in them. 

This witnessing technique was also shown to be useful to characterize provably total recursive functions in $I\Sigma_{k}$-fragments of $PA$ and their ordinal analysis (\cite{firstwitnessing} and \cite{Bussordinalanalysis}).
However, the arguments in this context have been treating only formulae having just one nontrivial block of $\exists$-quantifiers.
Here, by ``nontrivial block,'' we mean the block remaining after our modding out the part of formula which can be decided by oracles in particular settings.
For example, the proof of Theorem 7 in \cite{Bussordinalanalysis} and the proof of Theorem 5 in \cite{Buss}.
This suffices because the transfinite induction for $\Sigma_{n-1}$-formulae on $\omega_{m}$ is $\Pi_{n+1}$-conservative over the transfinite induction for $\Sigma_{n-2}$-formulae on $\omega_{m+1}$ (here, the base theory is $I\Delta_{0}$), and we can utilize it to reduce the quantifier complexity in concern. 
See \S 3.3 of \cite{Bussordinalanalysis} for details.
 
Viewing a number of works extending witnessing arguments to the formulae of arbitrary bounded-quantifier complexity in the realm of proof complexity and bounded arithmetics (see the first three sections of \cite{LiOliveira2023} for a comprehensive overview of the history and their game theoretic witnessing), it is natural to formulate a ``direct'' witnessing argument for the ordinal analysis of $PA$ and its fragments. 

In this article, we introduce a game notion $\mathcal{G}_{k}(\prec,h)$ for a given well-order $\prec$ and numbers $k, h \in \NN$. 
We prove that if $I\Sigma_{k}(X)$ proves transfinite induction for $\prec$ and a set $X$ with a (free-)cut-free sequent calculus proof of height $h$ and in free-variable normal form, then one player (called \textbf{Prover}) wins $\mathcal{G}_{k}(\prec,h)$ while the opponent (called \textbf{Delayer}) actually wins if the order type of $\prec$ is larger than an appropriate threshold below $\omega_{k+1}$, which depends on $h$.
The threshold converges from below to $\omega_{k+1}$, and hence the argument gives another proof of ordinal analysis of $I\Sigma_{k}$.
An interesting point here is how we obtain $\mathcal{G}_{k}$ for $k \geq 2$; $\mathcal{G}_{k}$ is defined as a game adding backtracking options to $\mathcal{G}_{k-1}$.
The notion can be regarded as a transfinite analogue of the game treated in \cite{GameforT22}, designed for analyzing the independence of the pigeonhole principle for $R$ over the bounded arithmetic $T^{2}_{2}(R)$, and the behavior of $\mathcal{G}_{k}$ we analyze in this article might be helpful for resolving the open questions in the paper.

The article is organized as follows:

\S \ref{Preliminaries} is for setup of notations and conventions.

In \S \ref{A game G1 for ISigma1(X) v.s. TI}, we introduce the game $\mathcal{G}_{1}(\prec,h)$, which is the base case of the whole of our game notion. 
We reprove ordinal analysis of $I\Sigma_{1}$ using $\mathcal{G}_{1}$.
Note that the argument in this section is just a transposition of the contents in \cite{firstwitnessing} and \cite{Bussordinalanalysis}; each formula we treat in this section is either $\Sigma_{1}(X)$, $\Pi_{1}(X)$, a subformula of transfinite induction, or an $X$-free false formula.

In \S \ref{A game Gk for ISigmak(X) v.s. TI}, we introduce the game $\mathcal{G}_{k}(\prec,h)$. 
The meat of the notion is that $\mathcal{G}_{k+1}$ is obtained by allowing \textbf{Prover} in $\mathcal{G}_{k}$ to backtrack a play and change a move in the past, bringing back the information obtained at the current position.
Based on $\mathcal{G}_{k}$, we reprove ordinal analysis of $I\Sigma_{k}$.
 
The last section is Appendix, considering the precise threshold of the order type of $|{\prec}|$ where the behavior of $\mathcal{G}_{1}(\prec,h)$ changes.

\section{Preliminaries}\label{Preliminaries}
In this article, we mainly work on first-order logic.
We adopt $\lnot, \land, \lor, \forall, \exists$ as logical symbols and consider only logical formulae of negation normal form.
Given a formula $\varphi$, $\overline{\varphi}$ denotes the canonical negation normal form of $\lnot \varphi$. It is called \deff{the complement of $\varphi$}.
Let $\mathcal{L}_{\NN}$ be the language collecting all the functions and predicates over the standard model $\NN$.
In particular, $\mathcal{L}_{\NN}$ includes predicate symbols for the equality $=$ and for the standard ordering $\leq$ of $\NN$, which we also denote by $=$ and $\leq$.
In particular, we can consider bounded formulae in terms of $\leq$.
Let $\Delta_0$ be the class of all the bounded $\mathcal{L}_\NN$-formulae.

For simplicity, in this article, we consider proper and strict versions of the classes $\Sigma_k$ and $\Pi_k$.
A class $p\Sigma_{i}$ is the collection of all the $\mathcal{L}_{\NN}$-formula $\varphi$ having the following form:
\[
\varphi \equiv \underbrace{\exists x_1 \forall x_2 \cdots}_{\mbox{exactly} \ i\  \mbox{-times}}\psi,
\]
 where $\psi$ is a maximal $\Delta_0$-subformula of $\varphi$.
 Similarly, the class $p\Pi_{i}$ is defined, switching the roles of $\exists$ and $\forall$.
 Note that $p\Sigma_{i}$ and $p\Sigma_{i'}$ are disjoint if $i \neq i'$, and similarly for $p\Pi_{i}$'s.
 We say $\varphi$ is $s\Sigma_{i}$ if it is $p\Sigma_{i^{\prime}}$ for some $0 \leq i^{\prime} \leq i$, or it is $p\Pi_{i^{\prime}}$ for some $0 \leq i^{\prime} < i$.
 Similarly for $s\Pi_{i}$.
 
 For a given formula $\varphi(\vec{x})$, its universal closure $\forall \vec{x}. \varphi(\vec{x})$ is denoted by $\forall \forall \varphi$.
For a formula $\varphi(\vec{x})$, an $\mathcal{L}_{\NN}$-structure $\mathcal{M}$, and an assignment $E$ in $\mathcal{M}$ covering all the variables in $\vec{x}$, $(\mathcal{M},E) \models \varphi(\vec{x})$ if and only if $\varphi(\vec{x})$ is true in $\mathcal{M}$ under $E$.

Let $Th(\NN)$ be the set of all the true $\mathcal{L}_{\NN}$-sentences in $\bigcup_{i=0}^{\infty}p\Sigma_{i}$.

We are interested in relativized arithmetics;
let $X$ be a fresh unary predicate symbol, and put $\mathcal{L}_\NN(X) := \mathcal{L}_{\NN} \cup \{X\}$.
Since $X$ is meant to represent an arbitrary subset of $\NN$, we denote $X(x)$ by ``$x \in X$'' for readability.

Let $\Delta_0(X)$ be the class of bounded formulae in the language $\mathcal{L}_{\NN}(X)$. 
The classes $p\Sigma_k(X)$, $p\Pi_k(X)$, $s\Sigma_k(X)$ and $s\Pi_k(X)$ are defined similarly as unrelativized ones.

Given $k \in \NN$, the theory $I\Sigma_k(X)$ is defined as:
\[I\Sigma_k(X) := Th(\NN)+p\Sigma_k(X)\mathchar`-IND.\]
Here, for a set $\Phi$ of $\mathcal{L}_{\NN}$-formulae, \textit{$\Phi$-IND} is the following axiom scheme:
\[\left(\varphi(0) \land \forall x. (\varphi(x) \rightarrow \varphi(x+1))\right) \rightarrow \forall x. \varphi(x).\]
Note that it is an extension of the usual $I\Sigma_k(X)$ although we restricted the definition of $\Sigma_k$ to $p\Sigma_k$ since $\mathcal{L}_\NN$ and $Th(\NN)$ include enough functions and their basic properties to transform general $\Sigma_k(X)$-formulae to $p\Sigma_k(X)$-formulae. 

Towards our descriptions of sequent calculus and a game, we clarify our treatment of \textit{trees}; to be precise, we are interested in edge-labeled rooted trees:
\begin{defn}\label{defofsequence}
 
 For $\sigma=(a_{1},\ldots, a_{h}) \in \omega^{h}$ and $k \geq 1$, set
 \begin{align*}
  \sigma_{k} := \begin{cases}
   a_{k} \quad &(k \leq h) \\
   -1 \quad &(k > h)
  \end{cases}.
 \end{align*}
 
 If $k \leq h$, then we define $\sigma_{\leq k}:=(\sigma_{1},\ldots, \sigma_{k})$.
 
  Given $\sigma \in \omega^{h}$ and $\tau \in \omega^{k}$, $\sigma*\tau$ denotes the concatenation $ (\sigma_{1}, \ldots, \sigma_{h}, \tau_{1}, \ldots, \tau_{k}) \in \omega^{h+k}$.

\end{defn}

\begin{rmk}
 We often identify $\omega$ as $\omega^{1}$ and abuse the notation.
 For example, we write $\sigma*k$ for $\sigma*(k)$.   
\end{rmk}

Note that we have used the ordinal notation $\omega$ instead of $\NN$ in Definition \ref{defofsequence}, although they could be identified as sets.
The purpose is to distinguish their \textit{types} or \textit{roles} in the arguments.
In this article, when we use $\NN$, we consider the standard model of the language $\mathcal{L}_{\NN}$.
For example, when we manipulate the values of closed $\mathcal{L}_{\NN}$-terms, we evaluate them in $\NN$. 
On the other hand, we use ordinals to describe (iterated sequences of) trees. 
For example, when we deal with a proof of sequent calculi, the underlying proof-tree is represented by a set of sequences of (finite) ordinals.

\begin{defn}
We set $\omega^{<\omega}:=\bigcup_{h =0}^{\infty} \omega^{h}$, that is, $\omega^{<\omega}$ is the set of all the finite sequences on $\omega$ (including the empty sequence $\emptyset$).
We equip it with \deff{the lexicographic order}:
For $v, w \in \omega^{<\omega}$,
\begin{align*}
v <_{lex} w : \Longleftrightarrow \exists k \in \omega.\ (v_{\leq k} = w_{\leq k} \ \& \ v_{k+1} < w_{k+1})
\end{align*}
\end{defn}

\begin{defn}
 For $\sigma, \tau \in \omega^{<\omega}$, $\sigma \subseteq \tau$ means that $\sigma$ is an initial segment of $\tau$ (or $\tau$ is an extension of $\sigma$).
 
 The length of $\sigma$ is denoted by $height(\sigma)$.
 
\end{defn}

\begin{defn}
\deff{A rooted tree} is a subset $T \subseteq \omega^{<\omega}$ such that:
\begin{enumerate}
\item $T \neq \emptyset$.
 \item $\sigma \subseteq \tau \in T \Longrightarrow \sigma \in T$.
\end{enumerate}
$\emptyset \in T$ is called \deff{the root of $T$}.

If $\sigma*k \in T$, then we say \deff{$\sigma*k$ is a child of $\sigma$}, and \deff{$\sigma$ is the parent of $\sigma*k$}. 
If $\sigma \in T$ does not have a child in $T$, that is, $\subseteq$-maximal in $T$, they are called \deff{leaves}.
We denote the set of all the leaves of $T$ by $L(T)$.

If $\sigma*k, \sigma*l \in T$, then they are said to be \deff{siblings}.

When $\sigma, \tau \in T$ satisfy $height(\sigma) = height(\tau)$, then we say \deff{$\sigma$ is left to $\tau$} when $\sigma <_{lex} \tau$.

For a finite tree $T$, put
\[height(T) := \max_{v \in T} height(v).\]
\end{defn}

\begin{eg}
If $\sigma*0 \in T$, then $\sigma*0$ must be the leftmost child of $\sigma$ in $T$.
\end{eg}

\begin{eg}
$height(T) \leq h \Longleftrightarrow T \subseteq \omega^{\leq h}$.
\end{eg}

Intuitively, the games we will present below has a ``snapshot'' $(T,\rho) \in [1,\omega^{h}] \times \mathcal{R}$ ``focusing on'' the current frontier $c(T)$,  where $[1,\omega^{h}], \mathcal{R}$ and $c(T)$ are defined as follows:

\begin{defn}
For ordinals $\alpha \leq \beta$, set $[\alpha,\beta] := \{x \in ON \mid \alpha \leq x \leq \beta\}$.
Similarly for other types of intervals such as $\left[\alpha , \beta\right[$.
\end{defn}

\begin{defn}\label{comb}
Let $h \in \omega$.
For each ordinal $\alpha \leq \omega^{h}$, we associate the following finite tree $T(\alpha)$ with it:
if $\alpha=\omega^{h}$, then set $T(\alpha):=\{\emptyset\}$.
If 
\[\alpha= \omega^{h-1}\cdot c_{h-1} + \cdots + \omega^{0}\cdot c_{0} \quad (c_{i} < \omega),\]
then $T(\alpha)$ is the collection of $\sigma \in \omega^{<\omega}$ such that $height(\sigma) \leq h$ and $\sigma$ is either empty or of the form $\sigma=(0,\cdots, 0, d)$, where $d \leq c_{h-height(\sigma)}$.

Obviously, the map $\alpha \mapsto T(\alpha)$ is injective on $[1,\omega^{h}]$, and its range is the set of isomorphic types of trees of height $\leq h$ in which every parent of a leaf is on the leftmost path and whose leftmost leaf has a sibling.
We call such a tree \deff{comb} (of height $\leq h$).
The inverse mapping is:
\[T \mapsto \sum_{v \in L(T) \setminus\{c(T)\}} \omega^{h-height(v)},\]
where the sum is \deff{the natural sum of base $\omega$}.

Since we want to manipulate the associated tree $T(\alpha)$ much more heavily than the original ordinal $\alpha$, \deff{we identify $\alpha$ and the comb $T(\alpha)$}, and we use symbols $T$ etc. for ordinals in $[1,\omega^{h}]$ (when $h$ is fixed).
\end{defn}

\begin{defn}
For a comb $T$, let $c(T)$ be \deff{the $<_{lex}$-minimal leaf} (``leftmost leaf'').
\end{defn}

\begin{rmk}
For a comb $T$, if $p \in T$ is the parent of a leaf, then $p \subsetneq c(T)$. 
\end{rmk}

\begin{rmk}
    Note that $height(c(T)) = height(T)$.
\end{rmk}


\begin{defn}\label{partialpredicate}
\deff{A partial predicate on $\NN$} is a function $D \rightarrow \{0,1\}$ whose domain $D$ is a subset of $\NN$.
  A partial predicate is \deff{finite} if and only if the domain is finite.
  Let $\mathcal{R}$ be the set of all the finite partial predicates on $\NN$.
  
  For $\rho,\rho' \in \mathcal{R}$, we say \deff{$\rho$ contradicts $\rho'$} when $\rho' \cup \rho$ is not a partial predicate.
  
  For $\rho \in \mathcal{R}$ and $Q \subseteq \NN$, \deff{$\rho$ covers $Q$} if and only if $Q \subseteq \dom(\rho)$.
\end{defn}

Besides, we define the following relatively strong notion of homomorphism to describe some winning strategies in the latter sections: 

\begin{defn}
Let $T_{1},T_{2}$ be trees. A map $h \colon T_{1} \rightarrow T_{2}$ is a \deff{homomorphism} if and only if the following holds:
if $v \in T_{1}$ is a child of $w \in T_{1}$, then $h(v)$ is a child of $h(w)$ in $T_{2}$.
Note that it automatically follows that $h$ preserves the relation $\subseteq$.
\end{defn}

Towards analysis of the game notions given in the following sections, we also set up another notation on ordinals: we denote the class of all the ordinals by $ON$.

\begin{defn}
For $\alpha,\beta \in ON$ and $k < \omega$, $\beta_{k}(\alpha) \in ON$ is inductively defined as follows:
\[\beta_{0}(\alpha) := \alpha, \ \beta_{k+1}(\alpha) := \beta^{\beta_{k}(\alpha)}.\]

\deff{We abbriviate $\omega_{k}(1)$ as $\omega_{k}$}.
\end{defn}

\section{A game $\mathcal{G}_1(h)$ for $I\Sigma_1(X)$ v.s. $TI(\prec)$}\label{A game G1 for ISigma1(X) v.s. TI}
In this section, we consider the base case of game-proof correspondence.
Following \cite{OrdinalAnalysis}, throughout this section, we fix a well-order (but not necessarily primitive recursive) $\prec$ on $\NN$.
Its order-type is denoted by $|{\prec}|$.
Note that there is a predicate symbol representing $\prec$ and $\preceq$ in $\mathcal{L}_\NN$.
We denote them by $\prec$ and $\preceq$, too.

Let $TI(\prec)$ be the following sequent consisting of $s\Sigma_2(X)$-formulae:
\[TI(\prec) :=\left\{ \forall y. \overline{y \in X}, \exists x_0\forall x. ( x_0 \in X \land (\overline{x \in X} \lor \overline{x \prec x_{0}}))\right\}.\]
Note that it is a natural prenex normal form of transfinite induction for $\prec$ and the complement of $X$ (in negation normal form).
\subsection{The game $\mathcal{G}_{1}$ and who wins}\label{The game G1 and who wins}
We begin with the description of the game.
\begin{defn}\label{DefG1}
For a parameter $h\in \omega$,
 $\mathcal{G}_{1}(\prec,h)$ is the following game (for intuitions, see Remark \ref{IntuitiononDefG1}):
  
 \begin{enumerate}
  \item Played by two players. We call them \textit{\textbf{Prover}} and \textit{\textbf{Delayer}}. 
  
  \item \textit{\textbf{A possible position}} is a pair $(T,\rho) \in [1,\omega^{h}] \times \mathcal{R}$ (cf. Definition \ref{comb} and \ref{partialpredicate}).
  Note that we identify the ordinal $T$ and its associated comb in a way of Definition \ref{comb}.
  
  For future convenience, let $\mathcal{P}_{1}$ be \deff{the set of all possible positions}, that is, $\mathcal{P}_{1} := [1,\omega^{h}] \times \mathcal{R}$.

   \item \textbf{Delayer} chooses the \textit{\textbf{initial position}} $(T_{0},\rho_{0})$, where $T_{0}=\omega^{h}$, or the rooted tree of height $0$, that is, consists only of the root $\emptyset$, $\dom(\rho_0)$ is a singleton $\{m_{0}\}$, and $\rho_0(m_{0})=1$.
   \item\label{transitioninG1} Now, we describe transitions between positions together with each player's options and judgment of the winner:
   suppose the current position is $(T,\rho)$.
   
   \begin{enumerate}
   \item\label{query} First, \textbf{Prover} plays a finite subset $Q \subseteq \NN$, and send it to \textbf{Delayer}.   
   \item\label{answer} \textbf{Delayer} plays a finite partial predicate $\rho^\prime$ covering $Q$, and send it back to \textbf{Prover}.
   $\rho'$ must satisfy the following (note that such $\rho'$ exists; for example, a constant function $0$ on $Q$):
   \begin{itemize}
    \item For all $q \in Q \cap (\rho^\prime)^{-1}(1)$, there exists $p \in (\rho^\prime)^{-1}(1)$ such that $p \prec q$.
    \end{itemize}
   If $\rho^\prime \cup \rho$ is not a partial predicate, \textit{(that is, $\rho'$ ``contradicts'' $\rho$,)} then the play ends and \textbf{Prover} wins.
   Otherwise, proceed as follows.
   
   \item\label{nextmove} \textbf{Prover} plays a pair $\langle o,b\rangle$, where $o \in \{0,1\}$ and $b \in \omega^{<\omega}$.
   If $o=0$, $b$ must be in $\omega^{1}$, which is identified as $\omega$.
   If $o=1$, $b$ must satisfy $b \subsetneq c(T)$.
   \item Depending on $o$, the next position is determined as follows:
    
    \begin{enumerate}
     \item If $o=0$, set $T':=T \cup \{c(T)*i \mid i \in [0,b]\}$.
     If $height(T') > h$, the game ends, and \textbf{Prover} loses.
     Otherwise, $(T',\rho'\cup \rho) \in \mathcal{P}_{1}$, and it is the next position.
     
    \item If $o=1$, consider $c(T) = b*k*\sigma$. Note that $\sigma$ may be empty.
    
     \quad If $b*(k+1) \not\in T$, the play ends and \textbf{Prover} loses.
     
     \quad Consider the case when $b*(k+1) \in T$. 
     Note that $b*(k+1)$ is actually a leaf of $T$ by definition of possible positions.
     
     Cut off the descendents of $b*k$ (including itself) and let $T'$ be the resulting tree:
     \[T':=T \setminus \{b*k*\tau \mid \tau \in \omega^{<\omega}\}.\]
     Then $(T',\rho'\cup\rho) \in \mathcal{P}_{1}$, and it is the next position.
     
    \end{enumerate}
    \end{enumerate}
 \end{enumerate}
 In the course of the game, sequences $\vec{P}=(P_{0},\ldots,P_{l})$ are generated, where $P_{0}=(T_{0},\rho_{0})$, and $P_{i+1}=(T',\rho')$ is a next position of $P_{i}=(T, \rho)$.
  Each sequence is called \deff{a play of $\mathcal{G}_{1}(\prec,h)$}.
 Let $\mathcal{S}_{1}$ be \deff{the set of all the plays of $\mathcal{G}_{1}(\prec,h)$}.
\end{defn}

\begin{rmk}\label{IntuitiononDefG1}
The intuitions behind each item of Definition \ref{DefG1} is as follows (the numbering below respects that of the previous definition):
\begin{enumerate}
\item \textbf{Delayer} pretends to have a subset $X \subseteq \NN$ violating $TI(\prec)$, which is of course impossible, and \textbf{Prover} wants to disprove it by querying finitely many data of $X$.
\item $T$ serves as a clock of the game. 
If the height exceeds $h$, the game is over and \textbf{Delayer} wins.
If \textbf{Delayer} comes not to be able to answer anymore before that, \textbf{Prover} wins.
$\rho$ is the record of the answers \textbf{Delayer} made so far.
\item Since \textbf{Delayer} pretends to have $X \subseteq \NN$ which is nonempty but does not have a minimum, \textbf{Delayer} starts the play by claiming ``$m_{0} \in X$,'' assuring $X$ is non-empty.
\item
\begin{enumerate}
 \item $Q$ amounts to \textbf{Prover}'s queries in this particular turn, namely, ``$x \in X$?'' for each $x \in Q$.
 \item $\rho'$ amounts to \textbf{Delayer}'s answers to the previous queries of \textbf{Prover}: $\rho(x)=1$ corresponds to ``$x \in X$'' and $\rho(x)=0$ corresponds to ``$x \not\in X$.''
 
 Furthermore, \textbf{Delayer} also claims ``$m \in X$'' for some $m$ such that $m \prec q$ for every $q \in Q$ with the answer $q \in X$, assuring $X$ violates $TI[\prec]$. 
 \item $o$ denotes the type of \textbf{Prover}'s move, and $b$ designates the precise transformation of $T$ to be excuted.
 \item  
 \begin{enumerate}
  \item If $o=0$, add $(b+1)$-many children of $c(T)$ to $T$, and let $T'$ be the resulting new tree.
  \item If $o=1$, then, we go back to $b$ in $T$ and proceed to the right.

 \end{enumerate}
\end{enumerate}
\end{enumerate}
\end{rmk}
In the rest of this section, we fix a parameter $h$ in order to reduce indices showing dependency on $h$.

\begin{defn}\label{DefO1}
For each position $P=(T,\rho) \in \mathcal{P}_{1}$, let
\begin{align}\label{defofO1}
 O_1(P) := T \leq \omega^{h}.
\end{align}
\end{defn}



First, we observe that the game above is determined.

\begin{lemma}\label{clockismonotone}
Let $h \in \omega$.
Suppose $(T',\rho')$ is a next position of $(T,\rho)$.
Then $T' < T$.
\end{lemma}

\begin{proof}
If \textbf{Prover} chose $o=0$ for the transition, then $T'$ is obtained by replacing one $\omega^{h-height(T)}$ in $T$ with $\omega^{h-height(T)-1} \cdot (b+1)$, where $b \in \omega$. (cf. Definition \ref{DefG1} and \ref{DefO1}.)
On the other hand, if \textbf{Prover} chose $o=1$, then $T'$ is a proper subsummation of $T$.

\end{proof}

\begin{cor}\label{G1determined}
For any $h \in \omega$, $\mathcal{G}_{1}(\prec,h)$ ends within finitely many steps, determining the winner.
\end{cor}

\begin{proof}
When the game ends, it always determines who is the winner. 
Therefore, it suffices to show that the game ends with finitely many transitions, which is an immediate corollary of Lemma \ref{clockismonotone}: if $\mathcal{G}_1(\prec,h)$ transitions from a position $P=(T,\rho)$ to $P'=(T',\rho'\cup\rho)$, then $O_1(P') < O_1(P)$.
\end{proof}

Now, we determine who is the winner depending on $|{\prec}|$.
Towards it, we first clarify the notion of strategies for $\mathcal{G}_{1}$:
below, for a set $S$, we denote the set of all the finite subsets of $S$ by $\mathfrak{P}_{fin}(S)$. 
\begin{defn}
Recall the definition of $\mathcal{S}_{1}$ in Definition \ref{DefG1}.
\deff{A strategy of \textbf{Prover} for $\mathcal{G}_{1}(\prec,h)$} is a pair $(f_{1},f_{2})$ such that:
\begin{itemize}
 \item $f_{1} \colon \mathcal{S}_{1} \rightarrow \mathfrak{P}_{fin}(\NN)$.
 \item $f_{2} \colon \mathcal{S}_{1}\times \mathfrak{P}_{fin}(\NN) \times \mathcal{R} \rightarrow \{0,1\}\times\omega^{<\omega}$.
\end{itemize}

\deff{A strategy of \textbf{Delayer} for $\mathcal{G}_{1}$} is a function $g \colon \mathcal{S}_{1}\times \mathfrak{P}_{fin}(\NN) \rightarrow \mathcal{R}$ with $g(\emptyset, \emptyset) \in \NN$.

$(f_{1},f_{2})$ is \deff{a winning strategy} if and only if \textbf{Prover} wins against any \textbf{Delayer}'s strategy $g$, that is, \textbf{Prover} wins if the play of $\mathcal{G}_{1}(\prec,h)$ is carried out as follows:
\textbf{Delayer} starts the play by claiming $m_{0}:=g(\emptyset, \emptyset)$.
Going through a play $\vec{P}=(P_{0},\ldots,P_{l})$, \textbf{Prover} casts $Q=f_{1}(\vec{P})$ as the next query, \textbf{Delayer} answers $\rho'=g(\vec{P},Q)$, and \textbf{Prover} chooses the option $\langle o,b\rangle = f_{2}(\vec{P},Q,\rho')$.

Similarly for \textbf{Delayer}'s \deff{winning strategy}.
\end{defn}

Now, we shall describe a winning strategy for \textbf{Delayer} when $|{\prec}|$ is large enough:
\begin{prop}\label{DelayerwinsG1}
Let $h > 0$. If $|{\prec}| > \omega^{h+1}+\omega$, then \textbf{Delayer} has a winning strategy for $\mathcal{G}_{1}(\prec,h)$.
\end{prop}
\begin{proof}
First, for $P \in \mathcal{P}_{1}$, set
\[\widetilde{O}_1(P) := \omega \cdot O_{1}(P).\]
Besides, below, we identify the natural numbers and their order types under $\prec$, which are ordinals below $|{\prec}|$.

 The following strategy of \textbf{Delayer} suffices:
 \begin{enumerate}
 \item Choose $\omega^{h+1}+\omega \mapsto 1$ as $\rho_0$ in the initial position.
  \item Suppose a play went through $\vec{P}=(P_{0},\ldots, P_{l})$, the current position $P_{l}$ is $(T,\rho)$, and, for any $x \in \rho^{-1}(1)$, $\widetilde{O}_1(T)+\omega \preceq x$ holds (where the sum in the LHS is the ordinal sum).
  Note that, if $l=0$, then the condition is indeed satisfied.
  
 Given \textbf{Prover}'s query $Q \subseteq \NN$, consider the following $\rho_1 \colon \dom(\rho) \cup Q \rightarrow \{0,1\}$:
 \begin{align*}
  \rho_1(q) := \begin{cases}
   \rho(q) \quad &(\mbox{if $q \in \dom(\rho)$}) \\
   0 \quad & (\mbox{otherwise})
  \end{cases}.
 \end{align*}
 Since $\rho_1$ is again finite, we have 
 \[\left[\widetilde{O}_1(T), \widetilde{O}_1(T) + \omega\right[  \setminus \dom(\rho_1) \neq \emptyset.\]
 Let $m$ be its minimum, and answer $\rho':=\rho_1 \cup \{m \mapsto 1\}$.
 \end{enumerate}

 Then \textbf{Delayer} survives this turn, and if $(T',\rho'\cup\rho)$ is the next position, the following holds:
 \[\min \left((\rho'\cup\rho)^{-1}(1)\right) \geq \widetilde{O}_1(T')+\omega.\]
 
 Note that if a play transitions from $P=(T,\rho)$ to $P'=(T',\rho'\cup\rho)$, then 
 \[\widetilde{O}_1(P) \geq \omega \cdot (O_1(P')+1) \geq \widetilde{O}_1(P')+\omega\]
 
 by Lemma \ref{clockismonotone}.
\end{proof}

The above observation is tight; see Proposition \ref{ProverwinsG1}.

\subsection{\textbf{Prover}'s winning strategy extracted from a proof}\label{Prover's winning strategy for G1}

Now, we connect $\mathcal{G}_{1}(\prec,h)$ to the provability of $TI(\prec)$ in $I\Sigma_{1}(X)$.
Towards it, we introduce the following \deff{forcing notion}:
\begin{defn}
Let $\rho$ be a finite partial predicate on $\NN$.
 Let $\varphi(\vec{x},X)$ be a $\Delta_{0}(X)$-formula and $E$ be an assignment of $\vec{x}$ with natural numbers.
 We write \deff{$(E,\rho) \Vdash \varphi$} to denote that any extension $\chi \colon \NN \rightarrow \{0,1\}$ of $\rho$ gives a model $((\NN,\chi),E)$ satisfying $\varphi(\vec{x},X)$.
\end{defn}
\begin{rmk}
 Note that, under $E$, the truth value of $\varphi$ is indeed decided by finitely many data of an interpretation of $X$ since $\varphi$ is $\Delta_0(X)$. 
\end{rmk}

Furthermore, we take a proof-theoretical approach to extract a winning strategy of \textbf{Prover} from a proof of $TI(\prec)$ in Theorem \ref{ProofIsStrategy1}, so we adopt the following version of (free-)cut-free one-sided sequent calculus towards the proof of Theorem \ref{ProofIsStrategy1}:

\begin{defn}\label{sequentcalculus}
\deff{A sequent} is a finite set of first-order formulae. 
Given a sequent $\Gamma$, its \deff{semantic interpretation} is the first-order sentence $\forall \forall \bigvee_{\varphi \in \Gamma}\varphi$. 
Given a structure, $\Gamma$ is said to be \deff{true} if and only if its semantic interpretation is satisfied in the structure.

We often denote sequents of the form 
\[\Gamma_{1} \cup \ldots \cup \Gamma_{k}\cup \{\varphi_{1},\ldots, \varphi_{m}\}\]
 by 
 \[\Gamma_{1}, \ldots, \Gamma_{k}, \varphi_{1},\ldots, \varphi_{m}.\]
\end{defn}

\begin{defn}\label{DefSequentCalculus}
For a term $t$, let $var(t)$ be the set of all the variables occurring in $t$. 
For a formula $\varphi$, let $fv(\varphi)$ be the set of all the free variables occurring in $\varphi$.
For a sequent $\Gamma$, let $fv(\Gamma):=\bigcup_{\varphi \in \Gamma} fv(\varphi)$.
\end{defn}
Now, one-sided sequent-calculus formulation of $I\Sigma_k(X)$ can be considered.
We adopt the conventions given in \cite{OrdinalAnalysis}, that is:
\begin{defn}
Given a sequent $S$ and a pair $\pi=(\tau,\lambda)$, where $\tau$ is a finite tree and $\lambda$ is a map on $\tau$ (i.e. vertex-labeling), $\pi$ is \deff{an $I\Sigma_{k}(X)$-derivation of $S$ in free variable normal form without redundancy} ($\pi \colon I\Sigma_{k}(X) \vdash S$ in short) if and only if the following hold:
\begin{enumerate}
 \item For each $v \in \tau$, $\lambda(v)$ is a sequent.
 \item $\lambda(\emptyset) = S$.
 \item For each $v \in \tau$, $\lambda(v)$ is derived from the labels of its children, that is, $(\lambda(v*k))_{v*k \in \tau}$ by applying one of the following derivation rules:

 \begin{itemize}
 \item Initial Sequent:
 \begin{prooftree}
 \AxiomC{}  \RightLabel{\quad (where $L$ is a literal)}
  \UnaryInfC{$\Gamma, L, \overline{L}$}
\end{prooftree}

 \item $\lor$-Rule:
   \begin{prooftree}
 \AxiomC{$\Gamma, \varphi_{i_{0}}$} \RightLabel{\quad (where $\varphi_{1} \lor \varphi_{2} \in \Gamma$, $i_{0}=1,2$, $\varphi_{i_{0}} \not\in\Gamma$.)}
  \UnaryInfC{$\Gamma$}
 \end{prooftree}

   \item $\exists$-Rule: 
   \begin{prooftree}
 \AxiomC{$\Gamma, \varphi(u)$}
    \UnaryInfC{$\Gamma$}
 \end{prooftree}
(where $\exists x. \varphi(x) \in \Gamma$, $u$ is an $\mathcal{L}_{\NN}$-term, $var(u) \subseteq fv(\Gamma \cup \varphi(u))=fv(\Gamma)$, and $\varphi(u) \not\in\Gamma$.)

 
 \item $\land$-Rule:
    \begin{prooftree}
 \AxiomC{$\Gamma, \varphi_{1}$}
 \AxiomC{$\Gamma, \varphi_{2}$}
   \RightLabel{\quad (where $\varphi_{1} \land \varphi_{2} \in \Gamma$, $\varphi_{1}\not\in\Gamma$,$\varphi_{2} \not\in\Gamma$.)}
  \BinaryInfC{$\Gamma$}
 \end{prooftree}

 \item $\forall$-Rule:
    \begin{prooftree}
 \AxiomC{$\Gamma, \varphi(a)$}
  \UnaryInfC{$\Gamma$}
 \end{prooftree}
(where $\forall x. \varphi(x) \in \Gamma$, $a \not \in fv(\Gamma)$, $\varphi(a) \not \in \Gamma$, $fv(\Gamma \cup \{\varphi(a)\}) \subseteq fv(\Gamma) \cup \{a\}$.
Note that $a$ might not occur in $\varphi(a)$.
The variable $a$ is called \deff{the eigenvaribale} of this rule.)


 \item True Sentence:
  \begin{prooftree}
 \AxiomC{$\Gamma, \overline{\varphi}$} \RightLabel{\quad (where $\varphi \in Th(\NN)$, and $\overline{\varphi} \not \in \Gamma$.)}
  \UnaryInfC{$\Gamma$}
 \end{prooftree}

 \item $p\Sigma_{k}(X)$-Induction:
     \begin{prooftree}
 \AxiomC{$\Gamma, \varphi(0)$}
 \AxiomC{$\Gamma, \overline{\varphi(a)},\varphi(a+1)$}
  \AxiomC{$\Gamma, \overline{\varphi(t)}$}
  \TrinaryInfC{$\Gamma$}
 \end{prooftree}
(where $t$ is an $\mathcal{L}_{\NN}$-term with $var(t) \subseteq fv(\varphi(t))$, $\varphi(x) \in p\Sigma_k(X)$, $\varphi(0) \not \in \Gamma$, $\varphi(t) \not\in\Gamma$, and
\[a \not \in fv(\Gamma \cup \{ \varphi(0)\})=fv(\Gamma \cup \{\overline{\varphi(t)}\})=fv(\Gamma),\quad fv(\Gamma\cup \{\overline{\varphi(a)}, \varphi(a+1)\}) = fv(\Gamma) \cup \{a\}.\]
 It automatically follows that $a \in fv(\varphi(a))$, and $\overline{\varphi(a)},\varphi(a+1) \not \in \Gamma$.
 The variable $a$ is called \deff{the eigenvariable} of this rule.)
 
 \item $p\Sigma_{k}(X)$-Cut:
      \begin{prooftree}
 \AxiomC{$\Gamma, \varphi$}
  \AxiomC{$\Gamma, \overline{\varphi}$}
  \BinaryInfC{$\Gamma$}
 \end{prooftree}
 (where $\varphi \in p\Sigma_k(X)$, $\varphi \not \in \Gamma$, $\overline{\varphi} \not \in \Gamma$, and $fv(\Gamma \cup \{ \varphi\})=fv(\Gamma \cup \{\overline{\varphi}\})=fv(\Gamma)$.)

\end{itemize}
\end{enumerate}
\end{defn}

\begin{rmk}
Note that, given $\pi=(\tau,\lambda) \colon I\Sigma_{k}(X) \vdash S$, a leaf $v=(v_{1},\ldots, v_{h})$ in $\tau$, and a formula $\varphi \in \lambda(v_{j}) \setminus S$ for some $j \in [1,h]$, then $\varphi$ is introduced by an Initial Sequent at $\lambda(v_{h})$, and there uniquely exists an index $i<j$ such that $\varphi \in \lambda(v_{i+1}) \setminus \lambda(v_{i})$. 
\end{rmk}

\begin{rmk}
Note that the above sequent calculus is of free-cut-free fashion; it admits cut for $p\Sigma_{k}(X)$-formulae only, which can be also regarded as instances of $\Sigma_{k}(X)$-Induction.
We separated $p\Sigma_{k}(X)$-Induction and $p\Sigma_{k}(X)$-Cut according to the occurrences of the eigenvariable $a$ in Definition \ref{DefSequentCalculus} because of technical convenience.
\end{rmk}

Now, assume we have an $I\Sigma_{1}(X)$-derivation $\pi=(\tau,\lambda)$ of $TI(\prec)$ in free variable normal form.
 Let $h := height(\tau)$.
 We extract a winning strategy of \textbf{Prover} in $\mathcal{G}_{1}(\prec,h)$ from $\pi$.
 Assume that the play so far is $\vec{P}=(P_{0},\ldots,P_{l})$, $P_{l}=(T,\rho)$ and $v=c(T)$.
 Inductively on $l$, we define: not only \textbf{Prover}'s strategy $(f_{1}(\vec{P}), f_{2}(\vec{P},Q,\rho'))$, but also the following auxiliary data:
  \begin{itemize}
    \item a homomorphism $V[\vec{P}] \colon T \rightarrow \tau$,
    \item a finite assignment $E[\vec{P}]$ for $fv(\lambda(V[\vec{P}](v)))$,
    \item a counterexample function 
     \[W[\vec{P}] \colon \lambda(V[\vec{P}](v) )\cap \bigcup_{k=1}^\infty p\Pi_k(X) \rightarrow \NN,\]
 \end{itemize}
 
 We design them fulfilling the following condition:
 \begin{defn}[The condition (\dag) for $f_{1},f_{2},V,E,W$]\label{dagar}
 We name the following conjunctive condition for $f_{1},f_{2},V,E,W$ as $(\dag)$:
 \begin{enumerate}
\item Let $w \subsetneq v$ and $\lambda(V[\vec{P}](w))=\Gamma$ is introduced by $p\Sigma_{1}(X)$-Cut:
 \begin{prooftree}
 \AxiomC{$\Gamma, \varphi$}
  \AxiomC{$\Gamma, \overline{\varphi}$}
  \BinaryInfC{$\Gamma$}
 \end{prooftree}
 
Suppose $w*n \subseteq v$.
 Then $n \in \{0,1\}$, and:
 \begin{itemize}
  \item if $n=0$, then $\lambda(V[\vec{P}](w*n))=\Gamma\cup\{ \varphi\}$ and $w*1 \in T$.
  \item if $n=1$, $\lambda(V[\vec{P}](w*n))=\Gamma\cup\{\overline{\varphi}\}$.
 \end{itemize}

 \item Let $w \subsetneq v$ and $\lambda(V[\vec{P}](w))=\Gamma$ is introduced by $p\Sigma_{1}(X)$-Induction:
      \begin{prooftree}
 \AxiomC{$\Gamma, \varphi(0)$}
 \AxiomC{$\Gamma, \overline{\varphi(a)},\varphi(a+1)$}
  \AxiomC{$\Gamma, \overline{\varphi(t)}$}
  \TrinaryInfC{$\Gamma$}
 \end{prooftree}
  
 Suppose $w*n \subseteq v$.
 Then $n \in [0,E[\vec{P}](t)+1]$ (note that $E[\vec{P}]$ covers all the free variables occurring in $t$ and evaluate it as a natural number), $w*l \in T$ for every $l \in [n,E[\vec{P}](t)+1]$, and:
 \begin{itemize}
  \item If $n=0$, then $\lambda(V[\vec{P}](w*n))=\Gamma\cup\{ \varphi(0)\}$.
  \item If $n=E[\vec{P}](t)+1$, $\lambda(V[\vec{P}](w*n))=\Gamma\cup\{ \overline{\varphi(t)}\}$.
  \item Otherwise, $\lambda(V[\vec{P}](w*n))=\Gamma\cup \{\overline{\varphi(a)},\varphi(a+1)\}$, and $E[\vec{P}](a)=n-1$.
 \end{itemize}
 
 \item Let $\varphi \in \lambda(V[\vec{P}](v))$.
According to the complexity of $\varphi$, the following hold:
 \begin{enumerate}
  \item If $\varphi$ is $X$-free, then $(E[\vec{P}],\rho) \Vdash \overline{\varphi}$ (or, equivalently, $(\NN, E[\vec{P}]) \models \overline{\varphi}$).
  Furthermore, if $\varphi$ is of the form $\forall x. \psi(x)$ where $\psi(x) \in \bigcup_{i=0}^{\infty}p\Sigma_i$, then 
  \[(E[\vec{P}],\rho) \Vdash \overline{\psi(W[\vec{P}](\varphi))}.\]  
  \item Otherwise, if $\varphi \in \Delta_{0}(X)$, then $(E[\vec{P}],\rho) \Vdash \overline{\varphi}$.
  \item Otherwise, if $\varphi \in p\Pi_{1}(X)$, then $(E[\vec{P}],\rho) \Vdash \overline{\psi(W[\vec{P}](\varphi))}$, where $\varphi \equiv \forall x. \psi(x)$.
\end{enumerate}
 \end{enumerate}
 \end{defn}

Now, we give the quantitative statement and prove it, which amounts to construct $f_{1},f_{2},V,E,W$ satisfying $(\dag)$:

\begin{thm}\label{ProofIsStrategy1}
Suppose $I\Sigma_1(X) \vdash TI(\prec)$.
Then there exists $h \in \omega$ such that \textbf{Prover} has a winning strategy for $\mathcal{G}_1(\prec,h)$.
\end{thm}

\begin{proof}
 By assumption, we have an $I\Sigma_{1}(X)$-derivation $\pi=(\tau,\lambda)$ of $TI(\prec)$ in free variable normal form.
 Let $h := height(\tau)$.
 We consider \textbf{Prover}'s winning strategy for $\mathcal{G}_1(\prec,h)$.
 Assume that the play so far is $\vec{P}=(P_{0},\ldots,P_{l})$, $P_{l}=(T,\rho)$ and $v=c(T)$.
 Inductively on $l$, we define the following $V,E,W,f_{1},f_{2}$ enjoying $(\dag)$: 
  \begin{itemize}
    \item a homomorphism $V[\vec{P}] \colon T \rightarrow \tau$,
    \item a finite assignment $E[\vec{P}]$ for $fv(\lambda(V[\vec{P}](v)))$,
    \item a counterexample function 
     \[W[\vec{P}] \colon \lambda(V[\vec{P}](v) )\cap \bigcup_{k=1}^\infty p\Pi_k(X) \rightarrow \NN,\]
     \item \textbf{Prover}'s strategy $(f_{1}(\vec{P}), f_{2}(\vec{P},Q,\rho'))$.
 \end{itemize}

 By relabeling the vertices of $\tau$, we may assume that $\tau$ has a canonical labeling so that:
 \[\sigma*k \in \tau \ \& \ k' \leq k \Longrightarrow \sigma*k' \in \tau.\]
 
 By subformula-property of cut-free sequent calculus, each formula $\varphi$ appearing in $\pi$ is either $s\Sigma_{1}(X)$, $s\Pi_{1}(X)$, a subformula of $TI(\prec)$, or $X$-free.
 
 Now, consider the base case $l=0$.
 $P_{0}$ is the initial position $(T_0,\rho_0)$, where $\rho_{0}=\{m_{0} \mapsto 1\}$.
  Set:
  \begin{itemize}
  \item $V[\vec{P}] \colon T_{0} \mapsto \tau; \emptyset \mapsto \emptyset$,
  \item $E[\vec{P}] := \emptyset$ since there is no free-variable in $\lambda(V[\vec{P}](\emptyset)) = TI(\prec)$,
  \item $W[\vec{P}]; (\forall y. \overline{y \in X}) \mapsto m_0$.
  \end{itemize}
  $f_{1}(\vec{P})$ and $f_{2}(\vec{P},Q,\rho')$ are described comprehensively in the following inductive way.
  
 Suppose we are at a play $\vec{P}=(P_{0},\ldots,P_{l})$,  $V[\vec{P}], E[\vec{P}], W[\vec{P}]$ are already defined, and, together with $P_{l}=(T,\rho)$ and $v=c(T)$, they satisfy the condition $(\dag)$. (The initial position trivially satisfies them.)
 
 We describe $f_{1}(\vec{P})$, $f_{2}(\vec{P},Q,\rho')$, and $V[\vec{P}P_{l+1}]$, $E[\vec{P}P_{l+1}]$, $W[\vec{P}P_{l+1}]$ for the next position $P_{l+1}$. We split cases by the rule deriving $\lambda(V[\vec{P}](v))$ in $\pi$.
  $\lambda(V[\vec{P}](v))$ is not an Initial Sequent since, if it was the case, then the literals $L$ and $\overline{L}$ in $\lambda(V[\vec{P}](v))$ should be both falsified by $(E[\vec{P}],\rho)$, which is absurd.
  
  Note that, in each case below, the condition $(\dag)$ remain satisfied: (Below, we omit most conditions imposed on the formulae and the sequents related to each rule. See Definition \ref{DefSequentCalculus} for reference)

 \begin{enumerate}
   \item\label{AxiomCase} The case when $\lambda(V[\vec{P}](v))=\Gamma$ is derived by True Sentence:
   \begin{prooftree}
 \AxiomC{$\Gamma, \overline{\varphi}$} \RightLabel{\quad (where $\varphi \in Th(\NN)$)}
  \UnaryInfC{$\Gamma$}
 \end{prooftree}
  Set $f_{1}(\vec{P}):=\emptyset$. 
 Let $\rho'$ be \textbf{Delayer}'s answer.
 Set
 \[f_{2}(\vec{P},\emptyset,\rho'):= \left\langle 0, 0 \right\rangle.\]
  \textit{(In words, \textbf{Prover} queries nothing and just put a child of $v$.)}
  For the next position $P_{l+1}=(T',\rho'\cup\rho)$:
  \begin{itemize}
   \item set $V[\vec{P}P_{{l+1}}]$ as the extension of $V[\vec{P}]$ which maps the new child of $v$ to the child of $V[\vec{P}](v)$,
 \item $E[\vec{P}P_{l+1}] := E[\vec{P}]$,
 \item 
 \begin{align*}
 W[\vec{P}P_{l+1}] := 
 \begin{cases}
     W[\vec{P}] \quad &\mbox{(if $\overline{\varphi} \not \in \bigcup_{i=1}^{\infty}p\Pi_{i}$)}\\
     W[\vec{P}] \cup \{\overline{\varphi} \mapsto w\} \quad &\mbox{(if $\overline{\varphi} \in \bigcup_{i=1}^{\infty}p\Pi_{i}$ and $w$ is the $\leq$-least witness of $\varphi$)}
 \end{cases}
\end{align*}
\end{itemize}
 \item The case when $\lambda(V[\vec{P}](v))=\Gamma$ is derived by $p\Sigma_{1}(X)$-Cut:
      \begin{prooftree}
 \AxiomC{$\Gamma, \varphi$}
  \AxiomC{$\Gamma, \overline{\varphi}$}
  \BinaryInfC{$\Gamma$}
 \end{prooftree}
 (where $\varphi \in p\Sigma_1(X)$.)

 Set $f_{1}(\vec{P}):=\emptyset$. 
 Let $\rho'$ be \textbf{Delayer}'s answer.
 Set
 \[f_{2}(\vec{P},\emptyset,\rho'):= \left\langle 0, 1\right\rangle.\]
 
 \textit{(In words, \textbf{Prover} queries nothing and just put two children at $v$. $c(T)$ moves to the left child.)}
 For the next position $P_{l+1}=(T',\rho'\cup\rho)$, set $V[\vec{P}P_{{l+1}}]$ as the extension of $V[\vec{P}]$ which maps $v*n$ ($n =0,1$) to $V[\vec{P}](v)*n$.
 Furthermore, set:
  \begin{align*}
     E[\vec{P}P_{l+1}] := E[\vec{P}],\ W[\vec{P}P_{l+1}] := W[\vec{P}].
 \end{align*}

  \item The case when $\lambda(V[\vec{P}](v))=\Gamma$ is derived by $p\Sigma_{1}(X)$-Induction:
     \begin{prooftree}
 \AxiomC{$\Gamma, \varphi(0)$}
 \AxiomC{$\Gamma, \overline{\varphi(a)},\varphi(a+1)$}
  \AxiomC{$\Gamma, \overline{\varphi(t)}$}
  \TrinaryInfC{$\Gamma$}
 \end{prooftree}
 ($\varphi(x) \in p\Sigma_1(X)$, and $a$ is the eigenvariable.)

 Set $f_{1}(\vec{P}):=\emptyset$. 
 Let $\rho'$ be \textbf{Delayer}'s answer.
 Set
 \[f_{2}(\vec{P},\emptyset,\rho'):= \left\langle 0, E[\vec{P}](t)+1\right\rangle.\]
 \textit{(In words, \textbf{Prover} queries nothing and put children of $v$ corresponding to the sequents obtained by expanding the above $p\Sigma_{1}(X)$-Induction to consecutive cuts.)}
 
 For the next position $P_{l+1}=(T',\rho'\cup\rho)$, set $V[\vec{P}P_{{l+1}}]$ as the extension of $V[\vec{P}]$ which maps $v*n$ ($n \in [0, E[\vec{P}](t)+1]$) to:
 \begin{itemize}
   \item $V[\vec{P}](v)*0$ if $n=0$. (Note that $\lambda(V[\vec{P}](v)*0) = \Gamma \cup\{ \varphi(0)\}$.)
   \item $V[\vec{P}](v)*2$ if $n=E[\vec{P}](t)+1$. (Note that $\lambda(V[\vec{P}](v)*2) = \Gamma \cup \{\overline{\varphi(t)}\}$.)
   \item $V[\vec{P}](v)*1$ otherwise. (Note that $\lambda(V[\vec{P}](v)*1) = \Gamma\cup\{\overline{\varphi(a)},\varphi(a+1)\}$.)
  \end{itemize}
  Furthermore, set:
  \begin{align*}
     E[\vec{P}P_{l+1}] := E[\vec{P}],\ W[\vec{P}P_{l+1}] := W[\vec{P}].
 \end{align*}

  \item $\lor$-Rule: 
   \begin{prooftree}
 \AxiomC{$\Gamma, \varphi_{i_{0}}$} \RightLabel{\quad (where $\varphi_{1} \lor \varphi_{2} \in \Gamma$, $i_{0}=1,2$, $\varphi_{i_{0}} \not\in\Gamma$.)}
  \UnaryInfC{$\Gamma$}
 \end{prooftree}
 
 By definition of $s\Sigma_k$, $\varphi_{1}\lor \varphi_2$ is $\Delta_0(X)$.
 Hence, it is already falsified by $(E[\vec{P}], \rho)$.
 Define $f_{1},f_{2},V[\vec{P}P_{l+1}], E[\vec{P}P_{l+1}], W[\vec{P}P_{l+1}]$ similarly to Case \ref{AxiomCase}. 
 Note that $\varphi_{i_{0}}$ is $\Delta_{0}(X)$ and therefore $W[\vec{P}P_{l+1}]:=W[\vec{P}]$ in this case.
 
  \item $\land$-Rule:
    \begin{prooftree}
 \AxiomC{$\Gamma, \varphi_{0}$}
 \AxiomC{$\Gamma, \varphi_{1}$}
   \RightLabel{\quad (where $\varphi_{1} \land \varphi_{2} \in \Gamma$, $\varphi_{1}\not\in\Gamma$,$\varphi_{2} \not\in\Gamma$.)}
  \BinaryInfC{$\Gamma$}
 \end{prooftree}

 By definition of $s\Sigma_k$, $\varphi_{0}\land \varphi_1$ is $\Delta_0(X)$.
 Hence, it is already falsified by $(E[\vec{P}], \rho)$.
 Let $i$ be the least index such that $\varphi_i$ is falsified.
 
 We define $f_{1},f_{2},V[\vec{P}P_{l+1}], E[\vec{P}P_{l+1}], W[\vec{P}P_{l+1}]$ similarly to Case \ref{AxiomCase} except we set $V[\vec{P}P_{l+1}](v*0):=V[\vec{P}](v)*i$.
 
 \item $\forall$-Rule:
    \begin{prooftree}
 \AxiomC{$\Gamma, \varphi(a)$}
  \UnaryInfC{$\Gamma$}
 \end{prooftree}
(where $\forall x. \varphi(x) \in \Gamma$, and $a$ is the eigenvariable.)

By (\dag) of Induction Hypothesis, we have $(E[\vec{P}],\rho) \Vdash \overline{\varphi(W[\vec{P}](\forall x.\varphi(x)))}$.
Note that $\forall x. \varphi(x) \in s\Pi_{1}(X)$ by subformula-property of $\pi$.

 We define $f_{1},f_{2},V[\vec{P}P_{l+1}], E[\vec{P}P_{l+1}], W[\vec{P}P_{l+1}]$ similarly to Case \ref{AxiomCase} except we set 
 \begin{align*}
     E[\vec{P}P_{l+1}] &:= E[\vec{P}] \sqcup \{a \mapsto W[\vec{P}](\forall x.\varphi(x))\}.
 \end{align*}
 Note that $a$ is out of $\dom (E[\vec{P}])$.

    \item $\exists$-Rule: 
   \begin{prooftree}
 \AxiomC{$\Gamma, \varphi(u)$}
  \UnaryInfC{$\Gamma$}
 \end{prooftree}
 (where $\exists x. \varphi(x) \in \Gamma$.)
 
 We split cases according to the form of $\exists x. \varphi(x)$.
  \begin{enumerate}
  \item\label{alreadydetermined} First we consider the case when $\exists x. \varphi(x)$ is $X$-free or a $\Delta_{0}(X)$-formula.
  Then the definitions are analogous to Case \ref{AxiomCase}.
   \item\label{subformulaofTI} Next, we consider the case when 
   \begin{align*}
       \exists x. \varphi(x) \equiv \exists x_0\forall x_1. ( x_0 \in X \land (\overline{x_1 \in X} \lor \overline{x_{1} \prec x_{0}})) \ i.e. \\
       x \equiv x_0 \ \& \
       \varphi(x) \equiv \forall x_1. ( x_0 \in X \land (\overline{x_1 \in X} \lor \overline{x_{1} \prec x_{0}}))
   \end{align*}
   In this case, we set $f_1(\vec{P}) := \{E[\vec{P}](u)\}$.
 If $\rho'$ is \textbf{Delayer}'s answer, then either:
 \begin{itemize}
     \item $\rho'(E[\vec{P}](u)) =0$ or
     \item $\rho'(E[\vec{P}](u)) =1$ and there exists $m \prec E[\vec{P}](u)$ such that $\rho'(m)=1$.
 \end{itemize}
 Let $w$ be $0$ in the first case and be the $\leq$-minimum $m$ in the latter case.
 Set
 \[f_{2}(\vec{P},f_1(\vec{P}),\rho'):= \left\langle 0, 0 \right\rangle.\]
 For the next position $P_{l+1}=(T',\rho'\cup\rho)$, set $V[\vec{P}P_{{l+1}}]$ as the extension of $V[\vec{P}]$ which maps the new child of $v$ to the child of $V[\vec{P}](v)$.
 Define
 \begin{align*}
     E[\vec{P}P_{l+1}] := E[\vec{P}],\
     W[\vec{P}P_{l+1}] := W[\vec{P}] \sqcup \{\varphi(u) \mapsto w\}.
 \end{align*}
 Note that $\varphi(u) \not \in \Gamma$ by Definition \ref{DefSequentCalculus}.

   \item Otherwise, $\exists x. \varphi(x)$ is $p\Sigma_1(X)$ and eliminated by $p\Sigma_1(X)$-Induction or $p\Sigma_{1}(X)$-Cut in $\pi$. 
   The latter case is analogous to and simpler than the former case, so we focus on the former case.
   
   In particular, $\varphi(u)$ is $\Delta_0(X)$, and therefore there exists the minimum finite subset $Q\subseteq \NN$ such that any extension of $\rho$ covering $Q$ determines the truth value of $\varphi(u)$ under the assignment $E[\vec{P}]$.
  Set $f_1(\vec{P}) := Q$.
  
 If $\rho'$ is \textbf{Delayer}'s answer, then either:
\[(E[\vec{P}], \rho') \Vdash \varphi(u) \ \mbox{or} \ (E[\vec{P}], \rho') \Vdash \overline{\varphi(u)}.\]
In the latter case, set $f_{2}(\vec{P},f_1(\vec{P}),\rho'):= \left\langle 0, 0 \right\rangle$, and define $V[\vec{P}P_{{l+1}}]$, $E[\vec{P}P_{{l+1}}]$, $W[\vec{P}P_{{l+1}}]$ similarly to Case \ref{AxiomCase}.

 In the case when $(E[\vec{P}], \rho') \Vdash \varphi(u)$, we finally use the option $o=1$.
 By assumption, there exists $w \subsetneq v$ such that $\lambda(V[\vec{P}](w))=\Delta$ is derived by $p\Sigma_{1}(X)$-Induction:
 \begin{prooftree}
 \AxiomC{$\Delta, \psi(0)$}
 \AxiomC{$\Delta, \overline{\psi(a)},\psi(a+1)$}
  \AxiomC{$\Delta, \overline{\psi(t)}$}
  \TrinaryInfC{$\Delta$}
 \end{prooftree}
  
 and $\exists x. \varphi(x)$ is $\psi(0)$ or $\psi(a+1)$.
 
 We first consider the case when $\exists x. \varphi(x) \equiv \psi(0)$.
 The following figure comprehends the situation (``@ $\sigma$'' indicates that the corresponding sequent is labelled at a vertex $\sigma$ in $\tau$):
 $$
\infer{\Delta \quad (@V[\vec{P}](w))}{
 \infer*{\Delta, \psi(0) \quad (@V[\vec{P}](w*0))}{
  \infer[\exists x. \varphi(x) \equiv \psi(0) \in \Gamma]{\Gamma \quad (@V[\vec{P}](v))}{
   \Gamma, \varphi(u) \quad (@V[\vec{P}](v)*0)}
  }
  &
  \infer*{\Delta, \overline{\psi(a)},\psi(a+1)}{
  }
   &
  \infer*{\Delta, \overline{\psi(t)}}{
  }
}
$$
 
 Since $(E[\vec{P}], \rho') \Vdash \varphi(u)$, the value $E[\vec{P}](u)$ serves as a counterexample of $\overline{\psi(0)}$.
  By $(\dag)$ of Induction Hypothesis, we have $w*0 \subseteq v$, and $w*1 \in T$.
  Thus, we define
  \[f_{2}(\vec{P}, Q,\rho'):=\langle 1, w \rangle.\]
  Note that \textbf{Prover} does not lose by choosing this option, and also
  \[\lambda(V[\vec{P}](w*1))= \Delta\cup \{\overline{\psi(a)},\psi(a+1)\}.\]
  
 Let $E'$ be the restriction of $E[\vec{P}]$ to $fv(\Delta\cup \{\psi(0)\})=fv(\Delta)$ and $W'$ be the restriction of $W[\vec{P}]$ to the formulae in $\Delta$. 
 For the next position $P_{l+1}=(T',\rho'\cup\rho)$, define
 \begin{align*}
  V[\vec{P}P_{l+1}] &:=V[\vec{P}]\restriction T'\\
     E[\vec{P}P_{l+1}] &:= E'\sqcup \{a \mapsto 0\} \\
     W[\vec{P}P_{l+1}] &:= W' \sqcup \{\overline{\psi(a)} \mapsto E[\vec{P}](u)\}.
 \end{align*}
  
 Next, we consider the case when $\exists x. \varphi(x) \equiv \psi(a+1)$.
 Let $w*n \subseteq v$.
 By $(\dag)$ of Induction Hypothesis, we have $n \in [1,E[\vec{P}](t)]$, $\lambda(V[\vec{P}](w*n))=\Delta\cup\{ \overline{\psi(a)},\psi(a+1)\}$ and $E[\vec{P}](a) = n-1$.

 If $n=E[\vec{P}](t)$, then we have $(E[\vec{P}],\rho') \Vdash \psi(t)$, witnessed by $x= E[\vec{P}](u)$.
 We will record this $x$ as a counterexample for the p$\Pi_{1}(X)$-formula $\overline{\psi(t)}$ below.
 
 Set
 \[f_{2}(\vec{P},f_1(\vec{P}),\rho'):= \left\langle 1, w \right\rangle.\]
 Note that $w*(n+1) \in T$ by assumption and therefore \textbf{Prover} does not lose.

 Let $E'$ be the restriction of $E[\vec{P}]$ to $fv(\Delta\cup \{\overline{\psi(t)}\})=fv(\Delta)$ and $W'$ be the restriction of $W[\vec{P}]$ to the formulae in $\Delta$. 
 For the next position $P_{l+1}=(T',\rho'\cup\rho)$, define
 \begin{align*}
  V[\vec{P}P_{l+1}] &:=V[\vec{P}]\restriction T'\\
     E[\vec{P}P_{l+1}] &:= E' \\
     W[\vec{P}P_{l+1}] &:= W' \sqcup \{\overline{\psi(t)} \mapsto E[\vec{P}](u)\}.
 \end{align*}
 If $n < E[\vec{P}](t)$,
 let $E'$ be the restriction of $E[\vec{P}]$ to $fv(\Delta)$.
 We have 
 \[(E'\sqcup\{a \mapsto n\}, \rho') \Vdash \psi(a),\]
  witnessed by $x= E[\vec{P}](u)$.

 Hence, set
 \[f_{2}(\vec{P},f_1(\vec{P}),\rho'):= \left\langle 1, w \right\rangle.\]
 Note that $w*(n+1) \in T$ by assumption and therefore \textbf{Prover} does not lose.
 Let $W'$ be the restriction of $W[\vec{P}]$ to the formulae in $\Delta$. 
 For the next position $P_{l+1}=(T',\rho'\cup\rho)$, set 
 \begin{align*}
 V[\vec{P}P_{l+1}] &:=V[\vec{P}]\restriction T',\\
     E[\vec{P}P_{l+1}] &:= E' \sqcup \{a \mapsto n\},\\
     W[\vec{P}P_{l+1}] &:= W' \sqcup \{\overline{\psi(a)} \mapsto E[\vec{P}](u)\}.
 \end{align*}

 \end{enumerate}
 \end{enumerate}
 
 This completes the description of \textbf{Prover}'s strategy. 
 Since \textbf{Prover} can continue the play as long as \textbf{Delayer} can make a move, \textbf{Prover}'s strategy $(f_{1},f_{2})$ is a winning one.
\end{proof}

Together with Proposition \ref{DelayerwinsG1}, the previous theorem implies the following:

\begin{cor}
   If $\pi=(\tau,\lambda) \colon I\Sigma_1(X) \vdash TI(\prec)$ and $|{\prec}| > \omega^{h+1}+\omega$, then $height(\pi) > h$.

   In particular, if $|{\prec}| \geq \omega^{\omega}$, then $I\Sigma_1(X) \not\vdash TI(\prec)$.
\end{cor}

Furthermore, instead of treating $\prec$ as a binary predicate symbol, we can use a parametrized well-order $\prec_p$, that is, a ternary predicate symbol $\prec(p,x,y)$.
In particular, if $|{\prec_p}|=\omega^{f(p)+1}+\omega+1$ for a function $f$, then the shortest $I\Sigma_{1}(X)$-derivation (in cut-free and free-variable normal form) of $TI(\prec_p)$ for a fixed $p \in \NN$ has size at least $f(p)$.
In this way, given an arbitrary growth rate, we can construct a family of first-order formulae whose proof-lengths majorize the given one while the sizes of the formulae themselves are $O(\log(p))$ if we regard the size of the symbol $\prec(p,x,y)$ is $O(1)$.\footnote{If we stick to a usual language of arithmetic, say, the language of ordered rings, we can still show analogous results for a primitive recursive ordinal $\prec$ (or $\prec_p$ and a primitive recursive function $f$ as above) by replacing the symbols by their definitions.}

\section{A game $\mathcal{G}_k(h)$ for $I\Sigma_k(X)$ v.s. $TI(\prec)$}\label{A game Gk for ISigmak(X) v.s. TI}

In this section, we present game notions $\mathcal{G}_{k+1}(\prec,h)$ corresponding to $I\Sigma_{k+1}(X)$-proofs of $TI(\prec)$ for $k \geq 1$.

In order to reduce complexity, we introduce the following notation: 
\begin{defn}\label{Defe}
Given a nonempty sequence $\sigma=(\sigma_{1},\ldots, \sigma_{l})$ ($l \geq 1$) in general, set $e(\sigma):=\sigma_{l}$.
\end{defn}

\subsection{The game $\mathcal{G}_{k}$ and who wins}

\begin{defn}\label{DefGk}
Fix parameters $\prec,h$.
Inductively on $k \geq 1$, we define the game notion $\mathcal{G}_{k}(\prec,h)$ together with families $\mathcal{S}_{k}(\prec,h)$, $\mathcal{P}_{k}(\prec,h)$.
For readability, in the following, we suppress the parameters $(\prec,h)$, and put a superscript $(i)$ on positions in $\mathcal{G}_{i}$.

$\mathcal{G}_{1}$, $\mathcal{S}_{1}$ and $\mathcal{P}_{1}$ are already defined in the previous subsection \ref{The game G1 and who wins}.

Now, for $k+1 \geq 2$,
let $\mathcal{P}_{k+1}$ be the set of all subsequences of sequences $\sigma \in \mathcal{S}_{k}$.

 $\mathcal{G}_{k+1}$ is the following game (see Remark \ref{IntuitiononDefGk}):
  
 \begin{enumerate}
  \item Played by two players.
  We call them \textit{\textbf{Prover}} and \textit{\textbf{Delayer}}. 

  \item \textit{\textbf{A possible position}} is an element of $\mathcal{P}_{k+1}$. 
      
   \item \textbf{Delayer} chooses $m_{0} \in \NN$.
    The \textit{\textbf{initial position}} is $P^{(k+1)}_{0}$; here, $P^{(j)}_{0} \in \mathcal{P}_{j}$ ($j =1,\ldots,k+1$) is inductively defined as follows: 
    \[P^{(1)}_{0}:=(T_0,\rho_{0}),\ \mbox{where $\rho_{0}=\{m_{0} \mapsto 1\}$,
    and $P^{(j+1)}_{0}:=(P^{(j)}_{0})$}.\]
   \item Now, we describe transitions between positions together with each player's options and judgment of the winner:
   suppose the current position is 
   \[P^{(k+1)}=(P^{(k)}_{0},\ldots,P^{(k)}_{l}) \in \mathcal{P}_{k+1} \quad (l \geq 0).\]
   \begin{enumerate}
   \item\label{queryk+1} First, \textbf{Prover} plays a finite subset $Q \subseteq \NN$, and send it to \textbf{Delayer}.   
   \item\label{answerk+1} \textbf{Delayer} plays a finite partial predicate $\rho^\prime$ exactly in the same way as Definition \ref{DefG1} (\ref{answer}).
    
    If $\rho^\prime$ contradicts $\rho$, then the play ends and \textbf{Prover} wins.
   Otherwise, proceed as follows.
   
   \item\label{nextmovek+1} \textbf{Prover} plays a pair $\langle o,b\rangle$, where $o \in [0,k+1]$, and $b \in \omega^{<\omega}$.
   If $o\neq 1$, $b$ must be a number.
   \item Depending on $\langle o,b\rangle$, the next position is determined as follows:
    
    \begin{enumerate}
     \item If $o \in [0,k]$, consider the next position of $P^{(k)}_l=e(P^{(k+1)})$ in $\mathcal{G}_k$ along $Q,\rho', \langle o,b \rangle$.
     
     If there is none, by induction, \textbf{Prover} loses $\mathcal{G}_k$.
     In this case, the play of $\mathcal{G}_{k+1}$ ends, and \textbf{Prover} loses.
     
     Otherwise, let $P^{(k)}_{l+1}$ be the next position of $P^{(k)}_{l}$ in $\mathcal{G}_{k}$.
     Then the next position in $\mathcal{G}_{k+1}$ is $(P^{(k)}_i)_{i=1}^{l+1} \in \mathcal{P}_{k+1}$.

    \item If $o=k+1$, first check whether $b \geq l+1$. (Recall that $l+1$ is the length of $P^{(k+1)}$.)
    If so, the play ends, and \textbf{Prover} loses. 
    Otherwise, we have a position $P^{(k)}_b$ in $\mathcal{G}_k$. 
    Consider the next position of $P^{(k)}_{b}$ in $\mathcal{G}_k$ along $Q,\rho',\langle 0,0\rangle$.
    If there is none, by induction, \textbf{Prover} loses $\mathcal{G}_k$.
    In this case, the play ends and \textbf{Prover} loses $\mathcal{G}_{k+1}$.
    Otherwise, let $Q^{(k)}_{b} \in \mathcal{P}_{k}$ be the next position.
    Now, set the next position in $\mathcal{G}_{k+1}$ as $(P^{(k)}_1,\ldots,P^{(k)}_{b-1},Q^{(k)}_{b})$.
    
    \end{enumerate}
    \end{enumerate}
 \end{enumerate}

 In the course of the play of $\mathcal{G}_{k+1}$, sequences on $\mathcal{P}_{k+1}$ are generated.
 We call each of them \deff{a play of $\mathcal{G}_{k+1}$}, and $\mathcal{S}_{k+1}$ denotes the set of all the plays of $\mathcal{G}_{k+1}$.
\end{defn}

\begin{rmk}\label{IntuitiononDefGk}
The intuitions behind each item in Definition \ref{DefGk} are as follows (the numbering of the following items respects that in the definition):
\begin{enumerate}
 \item The aims of \textbf{Prover} and \textbf{Delayer} are the same as Remark \ref{IntuitiononDefG1}.
 \item A position of $\mathcal{G}_{k+1}$ is a partial record of a play of $\mathcal{G}_{k}$.
 \item \textbf{Delayer} starts the play in the same way as Remark \ref{IntuitiononDefG1}.
 \item The game is basically the same as $\mathcal{G}_{k}$; positions in $\mathcal{G}_{k}$ are concatenated one by one to the record, which is a position of $\mathcal{G}_{k+1}$.
 
 However, \textbf{Prover} has another option now, corresponding to $o=k+1$: \textbf{Prover} can choose to backtrack the game record to the past position $P^{(k)}_{b}$ and restart the play from there by adding one child to the current frontier, bringing back the current $\rho'$ and replacing $P^{(k)}_{b}$ by the next position.  
\end{enumerate}
\end{rmk}

We first observe that $\mathcal{G}_{k+1}$ is determined.
\begin{defn}
Recall $\widetilde{O}_{1}(P^{(1)}):=\omega \cdot O_{1}(P^{(1)})$ in the proof of Proposition \ref{DelayerwinsG1}, where $O_{1}(P^{(1)})$ is defined in Definition \ref{DefO1}.

Inductively on $k \geq 1$, for each $P^{(k+1)}=(P^{(k)}_{i})_{i=0}^{l} \in \mathcal{P}_{k+1}$, define
\begin{align}\label{defofOk}
 \widetilde{O}_{k+1}(P^{(k+1)}) := \sum_{i=0}^{l} 2^{\widetilde{O}_{k}(P^{(k)}_{i})} + 2^{\widetilde{O}_{k}(P^{(k)}_{l})}.
\end{align}
Here, the sums in the RHS are all the natural sum of base $2$.
\end{defn}



\begin{lemma}\label{monotonicityforGk}
Let $k \geq 0$.
Assume that $Q^{(k+1)}$ is a next position of $P^{(k+1)}$ in $\mathcal{G}_{k+1}(\prec,h)$.
Then we have $\widetilde{O}_{k+1}(Q^{(k+1)})+\omega \leq \widetilde{O}_{k+1}(P^{(k+1)})$. 
Here, the sum in the LHS is the ordinal sum.
\end{lemma}

\begin{proof}
By induction on $k$.
The case when $k=0$ is already dealt with in the proof of Proposition \ref{DelayerwinsG1}.

Consider the case when $k \geq 1$.
Let $\langle o,b \rangle$ be the option \textbf{Prover} made to transition to $Q^{(k+1)}$.
Let $P^{(k+1)}=(P_{0}^{(k)}, \ldots, P_{l}^{(k)})$.
If $o \leq k$, $Q^{(k+1)}$ is of the form $(P_{0}^{(k)}, \ldots, P_{l+1}^{(k)})$.
Therefore, by the definition of $\widetilde{O}$ and induction hypothesis,
\begin{align*}
 \widetilde{O}_{k+1}(P^{(k+1)}) =& \sum_{i=0}^{l} 2^{\widetilde{O}_{k}(P^{(k)}_{i})} + 2^{\widetilde{O}_{k}(P^{(k)}_{l})}\\
 \geq & \sum_{i=0}^{l} 2^{\widetilde{O}_{k}(P^{(k)}_{i})} + 2^{\widetilde{O}_{k}(P^{(k)}_{l+1})+\omega}\\
 =& \sum_{i=0}^{l} 2^{\widetilde{O}_{k}(P^{(k)}_{i})} + 2^{\widetilde{O}_{k}(P^{(k)}_{l+1})}\cdot \omega \\
 \geq &\sum_{i=0}^{l} 2^{\widetilde{O}_{k}(P^{(k)}_{i})} + 2^{\widetilde{O}_{k}(P^{(k)}_{l+1})}\cdot 2 + \omega \\
 = & \widetilde{O}_{k+1}(Q^{(k+1)})+\omega.
\end{align*}
Note that $2^{\widetilde{O}_{k}(P^{(k)}_{l+1})} >0$ for the last inequality.

If $o=k+1$, then $Q^{(k+1)}$ is of the form $(P_{0}^{(k)}, \ldots, P_{b-1}^{(k)}, Q^{(k)}_{b})$, where $b\leq l$, and $Q^{(k)}_{b}$ is a next position of $P_{b}^{(k)}$.
Therefore, by the definition of $\widetilde{O}$ and induction hypothesis,
\begin{align*}
 \widetilde{O}_{k+1}(P^{(k+1)}) =& \sum_{i=0}^{l} 2^{\widetilde{O}_{k}(P^{(k)}_{i})} + 2^{\widetilde{O}_{k}(P^{(k)}_{l})}\\
 > & \sum_{i=0}^{b} 2^{\widetilde{O}_{k}(P^{(k)}_{i})}\\
 \geq & \sum_{i=0}^{b-1} 2^{\widetilde{O}_{k}(P^{(k)}_{i})} + 2^{\widetilde{O}_{k}(Q^{(k)}_{b})+\omega}\\
  \geq & \sum_{i=0}^{b-1} 2^{\widetilde{O}_{k}(P^{(k)}_{i})} + 2^{\widetilde{O}_{k}(Q^{(k)}_{b})}\cdot \omega\\
 \geq & \widetilde{O}_{k+1}(Q^{(k+1)})+\omega.
\end{align*}
\end{proof}

\begin{lemma}
For any $h \in \NN$, $\mathcal{G}_{k+1}(\prec,h)$ ends within finitely many steps, determining the winner.
\end{lemma}

\begin{proof}
By Lemma \ref{monotonicityforGk}.
\end{proof}

\begin{prop}\label{DelayerwinsGk+1}
Let $h > 0$. If $|{\prec}| >  2_{k}(\omega^{h+1}) \cdot 2 +\omega$, then \textbf{Delayer} has a winning strategy for $\mathcal{G}_{k+1}(\prec,h)$.
\end{prop}
\begin{proof}
The proof is analogous to Proposition \ref{DelayerwinsG1}. 
Just replace $\widetilde{O}_{1}$ with $\widetilde{O}_{k+1}$ and use Lemma \ref{monotonicityforGk}.
\end{proof}

Towards analysis of $\mathcal{G}_{k+1}$, we introduce the following notations:
below, we put superscripts $\vec{(\cdot)}^{(i)}$ on sequences of positions in $\mathcal{G}_{i}$.
Note that, subsequences of a play in $\mathcal{G}_{i}$, which are sequences of positions in $\mathcal{G}_{i}$, are, at the same time, positions of $\mathcal{G}_{i+1}$.
Thus, the both superscripts $\vec{(\cdot)}^{(i)}$ and $(i+1)$ are valid, and we separate the usages according to which ``type'' we consider at each point of the argument.
\begin{defn}
Set 
\begin{align*}
T_{1}&\colon \mathcal{P}_{1} \rightarrow [1,\omega^{h}];\ (T,\rho) \mapsto T,\\
R_{1}&\colon \mathcal{P}_{1} \rightarrow \mathcal{R};\ (T,\rho) \mapsto \rho.
\end{align*}
For $k+1 \geq 2$, set
\begin{align*}
T_{k+1}&\colon \mathcal{P}_{k+1} \rightarrow [1,\omega^{h}];\ \vec{P}^{(k)} \mapsto T_{k}(e(\vec{P}^{(k)})),\\
R_{k+1}&\colon \mathcal{P}_{k+1} \rightarrow \mathcal{R};\ \vec{P}^{(k)} \mapsto R_{k}(e(\vec{P}^{(k)})).
\end{align*}
(cf. Definition \ref{Defe})

\end{defn}

The following observations clarify the structure of a position of $\mathcal{G}_{k+1}$:

\begin{lemma}\label{structureofGk}
Let $k \geq 1$.
Let $\vec{P}^{(k+1)}=(P^{(k+1)}_{0},\ldots, P^{(k+1)}_{l_{k+1}})\in \mathcal{S}_{k+1}$.
  For $j =k,\ldots,1$, let $\vec{P}^{(j)}:=(P^{(j)}_{0}, \ldots, P^{(j)}_{l_{j}}):=e^{k+1-j}(\vec{P}^{(k+1)}) = e^{k-j}(P^{(k+1)}_{l_{k+1}})$, where $e^{K}$ is the iterated composition of $e$ with $K$-times.
  Note that $e^{0}$ is the identity.
  
 
 Consider $\mathcal{G}_{k+1}$, and suppose the play so far is $\vec{P}^{(k+1)}=(P^{(k+1)}_{0},\ldots, P^{(k+1)}_{l_{k+1}})$, where $P^{(k+1)}_{l_{k+1}}$ is the current position.
If \textbf{Prover} chooses an option $\langle o,s \rangle \in [0,k+1] \times \omega^{<\omega}$ and the next position $P^{(k+1)}_{\star}$ exists, then the following hold:
  \begin{enumerate}
   \item\label{basicoptions} If $o=0,1$, then $e^{k-j}(P^{(k+1)}_{\star})$ is a prolongation of $\vec{P}^{(j)}$ by a new element for any $j \in [1,k]$.
   \item If $o \in [2,k+1]$, then:
    \begin{enumerate}
    \item\label{backtrackingasbasicoptions} for $j \in [o,k]$, $e^{k-j}(P^{(k+1)}_{\star})$ is a prolongation of $\vec{P}^{(j)}$ by a new element.
    \item\label{backtracking} for $j =o-1 \in [1,k]$, $e^{k-j}(P^{(k+1)}_{\star})$ is of the form $(P^{(o-1)}_{0}, \ldots, P^{(o-1)}_{s-1}, Q^{(o-1)}_{s})$, where $s \leq l_{j}$ and $Q^{(o-1)}_{s}$ is a next position of $P^{(o-1)}_{s}$. (Note that $s$ is specified by \textbf{Prover} as the second component of $\langle o,s \rangle$.)
    \item\label{insideofbacktracking} for $j \in [1,o-2]$, $e^{k-j}(P^{(k+1)}_{\star})=e^{o-2-j}(Q^{(o-1)}_{s})$ is a prolongation of $e^{o-2-j}(P^{(o-1)}_{s})$ by a new element.
    \end{enumerate}
  \end{enumerate}
  
Furthermore, for each $j \in [1,k]$ and $r < l_{j}$, there exists $s \leq l_{k+1}$ such that $P_{r}^{(j)}=e^{k+1-j}(P^{(k+1)}_{s})$.

\end{lemma}

\begin{proof}
We prove the whole statement by induction on $k \geq 1$.

When $k=1$, $\mathcal{G}_{k+1}$ in concern is $\mathcal{G}_{2}$. 
Furthermore, \textbf{Prover} has only $o=0,1, 2$ as their options.
If \textbf{Prover} chooses $o=0,1$, then the next position $P^{(2)}_{\star}$ is a prolongation of $P^{(2)}_{l_{2}}$ following a play of $\mathcal{G}_{1}$.
If \textbf{Prover} chooses $o=2$, then the next position $P^{(2)}_{\star}$ has a form 
\[P^{(2)}_{\star} = (P^{(1)}_{0}, \ldots, P^{(1)}_{s-1}, Q^{(1)}_{s}),\]
where $P^{(2)}_{l_{2}}$ is in the form of 
\[P^{(2)}_{l_{2}} = (P^{(1)}_{0}, \ldots, P^{(1)}_{l_{1}}),\]
$s \leq l_{1}$, and $Q^{(1)}_{s}$ is a next position of $P^{(1)}_{s}$ in $\mathcal{G}_{1}$.

Next, we consider general $k > 1$.
Suppose the play so far is $\vec{P}^{(k+1)}=(P^{(k+1)}_{0},\ldots, P^{(k+1)}_{l_{k+1}})$, where $P^{(k+1)}_{l_{k+1}}$ is the current position.
Furthermore, suppose \textbf{Prover} chooses an option $\langle o,s \rangle \in [0,k+1] \times \omega^{<\omega}$ and the next position $P^{(k+1)}_{\star}$ exists.

  \begin{itemize}
   \item If $o \in [0,k]$, then $P^{(k+1)}_{\star}=e^{0}(P^{(k+1)}_{\star})$ is a prolongation of $P^{(k+1)}_{l_{k+1}}=e^{0}(P^{(k+1)}_{l_{k+1}})$ following a play of $\mathcal{G}_{k}$.
   In particular, $e(P^{(k+1)}_{\star})$ is a next position of $e(P^{(k+1)}_{l_{k+1}})$ in $\mathcal{G}_{k}$.
   Based on this observation, the items (\ref{basicoptions})-(\ref{insideofbacktracking}) follow immediately from Induction Hypothesis.
   

  \item Assume $o=k+1$.
  Then, by Definition \ref{DefGk}, $P^{(k+1)}_{\star}=e^{0}(P^{(k+1)}_{\star})$ has a form
  \[P^{(k+1)}_{\star}= (P^{(k)}_{0}, \ldots, P^{(k)}_{s-1}, Q^{(k)}_{s}),\]
  where $P^{(k+1)}_{l_{k+1}}=e^{0}(P^{(k+1)}_{l_{k+1}})$ is in the form of 
\[P^{(k+1)}_{l_{k+1}} = (P^{(k)}_{0}, \ldots, P^{(k)}_{l_{k}}),\]
$s\leq l_{k+1}$, and $Q^{(k)}_{s}$ is a next position of $P^{(k)}_{s}$ in $\mathcal{G}_{k}$ determined by \textbf{Prover}'s query $Q$, \textbf{Delayer}'s option $\rho'$, and \textbf{Prover}'s option $\langle 0,0 \rangle$.
This finishes the proof of item (\ref{backtracking}) in the lemma.

 Furthermore, for $j \in [1,o-2]=[1,k-1]$, applying Induction Hypothesis to $P^{(k)}_{s}$,$Q^{(k)}_{s}$, and the option $\langle 0,0 \rangle$, it follows that $e^{k-j}(P^{(k+1)}_{\star})=e^{k-1-j}(Q^{(k)}_{s})$ is a prolongation of $e^{k-1-j}(P^{(k)}_{s})$ by a new element.
\end{itemize}

\end{proof}

\subsection{\textbf{Prover}'s winning strategy extracted from a proof}

Similarly as \S\S \ref{Prover's winning strategy for G1}, we connect $\mathcal{G}_{k}(\prec,h)$ to the provability of $TI(\prec)$ in $I\Sigma_{k}(X)$ by Theorem \ref{ProofIsStrategyGk}.
To prove the theorem, we introduce the notion of \textit{ancestors}:
\begin{defn}
Let $\pi=(\tau,\lambda) \colon I\Sigma_{k}(X) \vdash S$.
We define two binary relations on $\{(v,\Phi) \mid v \in \tau,\ \Phi \in \lambda(v)\}$.
$(v,\Phi)$ is \deff{a direct ancestor}\footnote{Note that the terminology is not consistent with the graph-theoretical counterpart.
Here, we are following the terminologies in \cite{TakeutiProofTheory} and \cite{IntroductiontoProofTheory}.} of $(w,\Psi)$ if and only if the following hold:
\begin{enumerate}
 \item $v$ is a child of $w$. 
 \item $\Phi \equiv \Psi$, or, if the derivation rule applied to derive $\lambda(w)$ is one of those listed below, $\Phi$ and $\Psi$ are the following (to make the description concise, we omit diagrams to explain the formulae below. 
 See Definition \ref{sequentcalculus} for reference): 
 \begin{itemize}
  \item $\lor$-Rule: $\Phi \equiv \varphi_{i_{0}}$ and $\Psi \equiv \varphi_{1} \lor \varphi_{2}$.
   
   \item $\exists$-Rule: $\Phi \equiv \varphi(u)$ and $\Psi \equiv \exists x. \varphi(x)$.
    \item $\land$-Rule: $\Psi \equiv \varphi_{1} \land \varphi_{2}$.
     If $v=w*0$, $\Phi \equiv \phi_{1}$.
     If $v=w*1$, $\Phi \equiv \varphi_{2}$.
   
 \item $\forall$-Rule:
    $\Phi \equiv \varphi(a)$ and $\Psi \equiv \forall x. \varphi(x)$.
 
 
 \end{itemize}

\end{enumerate} 
The binary relation ``$(v,\Phi)$ is \deff{an ancestor} of $(w,\Psi)$'' is defined as the transitive closure of \textit{direct ancestor}. 
\end{defn}

Now, assume we have an $I\Sigma_{k+1}(X)$-derivation $\pi=(\tau,\lambda)$ of $TI(\prec)$ in free variable normal form.
 Let $h := height(\tau)$.
 We extract a winning strategy of \textbf{Prover} in $\mathcal{G}_{k+1}(\prec,h)$ from $\pi$.
 
 Assume that the play so far is $\vec{P}^{(k+1)}= (P^{(k+1)}_{0}, \ldots, P^{(k+1)}_{l_{k+1}})\in \mathcal{S}_{k+1}$.
  For $j \in [1,k]$, let 
    \[\vec{P}^{(j)}=(P^{(j)}_{0}, \ldots, P^{(j)}_{l_{j}}):=e^{k+1-j}(\vec{P}^{(k+1)})=e^{k-j}(P^{(k+1)}_{l_{k+1}}).\]
  Note that $\vec{P}^{(j)}=P^{(j+1)}_{l_{j+1}}$ for $j \in [1,k]$. 
  Furthermore, for $j \in [1,k+1]$ and $s \leq l_{j}$, we write $\vec{P}^{(j)}_{\leq s}$ to denote the subsequence $(P^{(j)}_{0}, \ldots, P^{(j)}_{s})$ of $\vec{P}^{(j)}$.
  For $s \leq l_{k+1}$, set $T_{s}:=T_{k+1}(P^{(k+1)}_{s})$ and $\rho_{s}:=R_{k+1}(P^{(k+1)}_{s})$.
  We denote $T_{l_{k+1}}$ and $\rho_{l_{k+1}}$ by $T$ and $\rho$ respectively.
  
 Inductively on $l_{k+1}$, we define, as in the proof of Theorem \ref{ProofIsStrategy1}, the following:
  \begin{itemize}
    \item a homomorphism $V[\vec{P}^{(k+1)}] \colon T \rightarrow \tau$ ,
    \item a finite assignment $E[\vec{P}^{(k+1)}]$ for $fv(\lambda(V[\vec{P}^{(k+1)}](c(T))))$
        \item a counterexample function
     \[W[\vec{P}^{(k+1)}] \colon \lambda(V[\vec{P}^{(k+1)}](c(T)))\cap \bigcup_{i=1}^\infty p\Pi_i(X) \rightarrow \NN,\]
     
    \item \textbf{Prover}'s strategy $(f_{1}(\vec{P}^{(k+1)}), f_{2}(\vec{P}^{(k+1)},Q,\rho'))$
 \end{itemize}
 
 We design them fulfilling the following condition (below, we omit the most conditions on formulae and sequents to which derivation rules are applied, and we just present important ones. For details, see Definition \ref{DefSequentCalculus}):
 \begin{defn}[Condition (\dag\dag) for $V,E,W, f_{1},f_{2}$]\label{doubledagar}
 $(\dag\dag)$ is the following conjunctive condition for $V,E,W$ and $f_{1},f_{2}$:
  \begin{enumerate}
 
 
  \item Let $\varphi \in \lambda(V[\vec{P}^{(k+1)}](c(T)))$.
According to the complexity of $\varphi$, the following hold:
 \begin{enumerate}
  \item If $\varphi$ is $X$-free, then 
  \[(E[\vec{P}^{(k+1)}],\rho) \Vdash \overline{\varphi}\quad  (\mbox{or, equivalently,}\ (\NN, E[\vec{P}^{(k+1)}]) \models \overline{\varphi}).\]
  Furthermore, if $\varphi$ is of the form $\forall x. \psi(x)$ where $\psi(x) \in \bigcup_{i=0}^{\infty} p\Sigma_i$, then 
  \[(E[\vec{P}^{(k+1)}],\rho) \Vdash \overline{\psi(W[\vec{P}^{(k+1)}](\varphi))}.\]  
  \item Otherwise, if $\varphi \in \Delta_{0}(X)$, then $(E[\vec{P}^{(k+1)}],\rho) \Vdash \overline{\varphi}$.
  \item Otherwise, if $\varphi \in p\Pi_{1}(X)$, then $(E[\vec{P}^{(k+1)}],\rho) \Vdash \overline{\psi(W[\vec{P}^{(k+1)}](\varphi))}$, where $\varphi \equiv \forall x. \psi(x)$.
\end{enumerate}

  \item\label{cutcase} Let $v=c(T)$, $w \subsetneq v$, and $\lambda(V[\vec{P}^{(k+1)}](w))=\Gamma$ is introduced by $p\Sigma_{k+1}(X)$-Cut:
      \begin{prooftree}
 \AxiomC{$\Gamma, \varphi$}
  \AxiomC{$\Gamma, \overline{\varphi}$}
  \BinaryInfC{$\Gamma$}
 \end{prooftree}
 (where $\varphi \in p\Sigma_{k+1}(X)$.)
 
 Suppose $w*n \subseteq v$.
 If $n=0$, $w*1 \in T$.
 Furthermore, the following hold:
 \begin{enumerate}
  \item\label{cuttotheleft} If $n=0$, then $\lambda(V[\vec{P}^{(k+1)}](w*n))=\Gamma\cup \{\varphi\}$.
  Furthermore, if $(V[\vec{P}^{(k+1)}](v), \forall x. \psi(x))$ is an ancestor of $(V[\vec{P}^{(k+1)}](w*n),\varphi)$ and $\forall x. \psi(x) \in p\Pi_{i}(X)$ ($i \in [2,k]$),
  then there exist $r<l_{k+2-i}$ and $s<l_{k+1}$ such that:
  \begin{enumerate}
   \item $P^{(k+2-i)}_{r}=e^{i-1}(P^{(k+1)}_{s})$.
   \item $\lambda(V[\vec{P}^{(k+1)}_{\leq s}](c(T_{s})))=\Delta$ is derived by $\exists$-rule of the following form:
     \begin{prooftree}
 \AxiomC{$\Delta, \overline{\Psi(u)}$}\RightLabel{($\exists x. \overline{\Psi(x)} \in \Delta$)}
    \UnaryInfC{$\Delta$}
 \end{prooftree}
 where $\NN \models E[\vec{P}^{(k+1)}_{\leq s}](\Psi(u))\leftrightarrow E[\vec{P}^{(k+1)}](\psi(W[\vec{P}^{(k+1)}](\forall x.\psi(x))))$.

  \item $V[\vec{P}^{(k+1)}](w)*1 \subseteq V[\vec{P}^{(k+1)}_{\leq s}](c(T_{s}))$.
   \end{enumerate}
   
   The following figure comprehends the situation (``@ $\sigma$'' indicates that the corresponding sequent is labelled at a vertex $\sigma$ in $\tau$. 
   Besides, the boxes indicate the ancestor relation.):
$$
\infer[\mbox{($p\Sigma_{k+1}$-Cut)}]{\Gamma \quad (@ V[\vec{P}^{(k+1)}](w))}{
	\infer*{\Gamma, \boxed{\varphi} \quad (@ V[\vec{P}^{(k+1)}](w)*0)}{
		\infer{\cdots, \boxed{\forall x. \psi(x)} \quad (@ V[\vec{P}^{(k+1)}](v))}{
		\vdots
		}
	}
	&
	\infer*{\Gamma, \overline{\varphi} \quad (@ V[\vec{P}^{(k+1)}](w)*1)}{
	    \infer[\mbox{($\exists$-rule)}]{\Delta \quad (@ V[\vec{P}_{\leq s}^{(k+1)}](c(T_{s})))}{
	     \infer{\Delta, \overline{\Psi(u)} \quad (@ V[\vec{P}_{\leq s}^{(k+1)}](c(T_{s}))*0)}{
	     \vdots}
	    }
	}
}
$$
   (where $u$ gives a candidate of counterexample of $\forall x. \psi(x)$.)
   
  \item\label{cuttotheright} If $n=1$, then $\lambda(V[\vec{P}^{(k+1)}](w*n))=\Gamma\cup \{\overline{\varphi}\}$.
   Furthermore, if $(V[\vec{P}^{(k+1)}](v), \forall x. \psi(x))$ is an ancestor of $(V[\vec{P}^{(k+1)}](w*n), \overline{\varphi})$ and $\forall x. \psi(x) \in p\Pi_{i}(X)$ ($i \in [2,k+1]$),
  then there exist $r<l_{k+2-i}$ and $s<l_{k+1}$ such that:
  \begin{enumerate}
  \item $P^{(k+2-i)}_{r}=e^{i-1}(P^{(k+1)}_{s})$.
   \item $\lambda(V[\vec{P}^{(k+1)}_{\leq s}](c(T_{s})))=\Delta$ is derived by $\exists$-rule of the following form:
     \begin{prooftree}
 \AxiomC{$\Delta, \overline{\Psi(u)}$}\RightLabel{($\exists x. \overline{\Psi(x)} \in \Delta$)}
    \UnaryInfC{$\Delta$}
 \end{prooftree}
 where $\NN \models E[\vec{P}^{(k+1)}_{\leq s}](\Psi(u))\leftrightarrow E[\vec{P}^{(k+1)}](\psi(W[\vec{P}^{(k+1)}](\forall x.\psi(x))))$.
 \item $V[\vec{P}^{(k+1)}](w)*0 \subseteq V[\vec{P}^{(k+1)}_{\leq s}](c(T_{s}))$.
   \end{enumerate}
    The following figure comprehends the situation:
$$
\infer[\mbox{($p\Sigma_{k+1}$-Cut)}]{\Gamma \quad (@ V[\vec{P}^{(k+1)}](w))}{
	\infer*{\Gamma, \varphi \quad (@ V[\vec{P}^{(k+1)}](w)*0)}{
		\infer[\mbox{($\exists$-rule)}]{\Delta \quad (@ V[\vec{P}_{\leq s}^{(k+1)}](c(T_{s})))}{
	     \infer{\Delta, \overline{\Psi(u)} \quad (@ V[\vec{P}_{\leq s}^{(k+1)}](c(T_{s}))*0)}{
	     \vdots}
	    }		
	}
	&
	\infer*{\Gamma, \boxed{\overline{\varphi}} \quad (@ V[\vec{P}^{(k+1)}](w)*1)}{
	    \infer{\cdots, \boxed{\forall x. \psi(x)} \quad (@ V[\vec{P}^{(k+1)}](v))}{
		\vdots
	    }
	   }
}
$$
  \end{enumerate}
(where $u$ gives a candidate of counterexample of $\forall x. \psi(x)$.)
 
 \item\label{doubledagarInduction} Let $v=c(T)$, $w \subsetneq v$, and $\lambda(V[\vec{P}^{(k+1)}](w))=\Gamma$ is introduced by $p\Sigma_{k+1}(X)$-Induction:
      \begin{prooftree}
 \AxiomC{$\Gamma, \varphi(0)$}
 \AxiomC{$\Gamma, \overline{\varphi(a)},\varphi(a+1)$}
  \AxiomC{$\Gamma, \overline{\varphi(t)}$}
  \RightLabel{\quad (where $\varphi \in p\Sigma_{k+1}(X)$, $a$ is an eigenvariable)}
  \TrinaryInfC{$\Gamma$}
 \end{prooftree}

 Suppose $w*n \subseteq v$.
 Then $n \in [0,E[\vec{P}^{(k+1)}](t)+1]$, $w*l \in T$ for every $l \in [n,E[\vec{P}^{(k+1)}](t)]$.
 Furthermore, if $E[\vec{P}^{(k+1)}](t) =0$, the two subitems (\ref{cuttotheleft})(\ref{cuttotheright}) of the previous item \ref{cutcase} 
 hold, replacing $\varphi$ with $\varphi(0)$, $\overline{\varphi}$ with $\overline{\varphi(t)}$, and ``$V[\vec{P}^{(k+1)}](w)*1$'' with ``$V[\vec{P}^{(k+1)}](w)*2$.''
 
 When $E[\vec{P}^{(k+1)}](t) \geq 1$, the following hold:
 \begin{enumerate}
  \item If $n=0$, then $\lambda(V[\vec{P}^{(k+1)}](w*n))=\Gamma\cup \{\varphi(0)\}$.
  Furthermore, if $(V[\vec{P}^{(k+1)}](v), \forall x. \psi(x))$ is an ancestor of $(V[\vec{P}^{(k+1)}](w*0),\varphi(0))$ and $\forall x. \psi(x) \in p\Pi_{i}(X)$ ($i \in [2,k]$),
  then there exist $r<l_{k+2-i}$ and $s<l_{k+1}$ such that:
  \begin{enumerate}
   \item $P^{(k+2-i)}_{r}=e^{i-1}(P^{(k+1)}_{s})$.
   \item $\lambda(V[\vec{P}^{(k+1)}_{\leq s}](c(T_{s})))=\Delta$ is derived by $\exists$-rule of the following form:
     \begin{prooftree}
 \AxiomC{$\Delta, \overline{\Psi(u)}$}\RightLabel{($\exists x. \overline{\Psi(x)} \in \Delta$)}
    \UnaryInfC{$\Delta$}
 \end{prooftree}
 where $\NN \models E[\vec{P}^{(k+1)}_{\leq s}](\Psi(u))\leftrightarrow E[\vec{P}^{(k+1)}](\psi(W[\vec{P}^{(k+1)}](\forall x.\psi(x))))$.
  \item $V[\vec{P}^{(k+1)}](w)*1 \subseteq V[\vec{P}^{(k+1)}_{\leq s}](c(T_{s}))$.
 \item $E[\vec{P}^{(k+1)}_{\leq s}](a)=0$.
   \end{enumerate}
   The following figure comprehends the situation:
   $$
\infer{\Gamma \quad (@ V[\vec{P}^{(k+1)}](w))}{
	\infer*{\Gamma, \boxed{\varphi(0)} \quad (@ V[\vec{P}^{(k+1)}](w)*0)}{
		\infer{\cdots, \boxed{\forall x. \psi(x)} \quad (@ V[\vec{P}^{(k+1)}](v))}{
		\vdots
		}
	}
	&
	\infer*{\Gamma, \overline{\varphi(a)}, \varphi(a+1) \quad (@ V[\vec{P}^{(k+1)}](w)*1)}{
	    \infer[\mbox{($\exists$-rule)}]{\Delta \quad (@ V[\vec{P}_{\leq s}^{(k+1)}](c(T_{s})))}{
	     \infer{\Delta, \overline{\Psi(u)} \quad (@ V[\vec{P}_{\leq s}^{(k+1)}](c(T_{s}))*0)}{
	     \vdots}
	    }
	}
	&
	\infer*{\Gamma, \overline{\varphi(t)}}{
	}
}
$$
   (where $u$ gives a candidate of counterexample of $E[\vec{P}^{(k+1)}](\forall x. \psi(x))$ under $E[\vec{P}^{(k+1)}_{\leq s}]$.)
   
  \item If $n=E[\vec{P}^{(k+1)}](t)+1$, then 
  \[\lambda(V[\vec{P}^{(k+1)}](w*n))=\lambda(V[\vec{P}^{(k+1)}](w)*2)=\Gamma\cup \{\overline{\varphi(t)}\}.\]
   Furthermore, if $(V[\vec{P}^{(k+1)}](v), \forall x. \psi(x))$ is an ancestor of $(V[\vec{P}^{(k+1)}](w*n), \overline{\varphi(t)})$ and $\forall x. \psi(x) \in p\Pi_{i}(X)$ ($i \in [2,k+1]$),
  then there exist $r<l_{k+2-i}$ and $s<l_{k+1}$ such that:
  \begin{enumerate}
  \item $P^{(k+2-i)}_{r}=e^{i-1}(P^{(k+1)}_{s})$.
   \item $\lambda(V[\vec{P}^{(k+1)}_{\leq s}](c(T_{s})))=\Delta$ is derived by $\exists$-rule of the following form:
     \begin{prooftree}
 \AxiomC{$\Delta, \overline{\Psi(u)}$}\RightLabel{($\exists x. \overline{\Psi(x)} \in \Delta$)}
    \UnaryInfC{$\Delta$}
 \end{prooftree}
 where $\NN \models E[\vec{P}^{(k+1)}_{\leq s}](\Psi(u))\leftrightarrow E[\vec{P}^{(k+1)}](\psi(W[\vec{P}^{(k+1)}](\forall x.\psi(x))))$.
 \item $V[\vec{P}^{(k+1)}](w)*1 \subseteq V[\vec{P}^{(k+1)}_{\leq s}](c(T_{s}))$.
 \item $E[\vec{P}^{(k+1)}_{\leq s}](a)=E[\vec{P}^{(k+1)}](t)$.
   \end{enumerate}
  $$
\infer{\Gamma \quad (@ V[\vec{P}^{(k+1)}](w))}{
   \infer*{\Gamma, \varphi(0)}{
   }
   &
	\infer*{\Gamma, \overline{\varphi(a)},\varphi(a+1) \quad (@ V[\vec{P}^{(k+1)}](w)*1)}{
		\infer[\mbox{($\exists$-rule)}]{\Delta \quad (@ V[\vec{P}_{\leq s}^{(k+1)}](c(T_{s})))}{
	     \infer{\Delta, \overline{\Psi(u)} \quad (@ V[\vec{P}_{\leq s}^{(k+1)}](c(T_{s}))*0)}{
	     \vdots}
	    }		
	}
	&
	\infer*{\Gamma, \boxed{\overline{\varphi(t)}} \quad (@ V[\vec{P}^{(k+1)}](w)*2)}{
	    \infer{\cdots, \boxed{\forall x. \psi(x)} \quad (@ V[\vec{P}^{(k+1)}](v))}{
		\vdots
	    }
	   }
}
$$
   (where $u$ gives a candidate of counterexample of $E[\vec{P}^{(k+1)}](\forall x. \psi(x))$ under $E[\vec{P}^{(k+1)}_{\leq s}]$.)

  \item Otherwise, 
  \[\lambda(V[\vec{P}^{(k+1)}](w*n))=\lambda(V[\vec{P}^{(k+1)}](w)*1)=\Gamma \cup\{ \overline{\varphi(a)},\varphi(a+1)\},\]
   and $E[\vec{P}^{(k+1)}](a)=n-1$.
  
   Consider the case when $(V[\vec{P}^{(k+1)}](v), \forall x. \psi(x) )$ is an ancestor of \\
   $(V[\vec{P}^{(k+1)}](w*n), \overline{\varphi(a)})$ and $\forall x. \psi(x) \in p\Pi_{i}(X)$ ($i \in [2,k+1]$):
     $$
\infer{\Gamma \quad (@ V[\vec{P}^{(k+1)}](w))}{
   \infer*{\Gamma, \varphi(0)}{
   }
   &
	\infer*{\Gamma, \boxed{\overline{\varphi(a)}},\varphi(a+1) \quad (@ V[\vec{P}^{(k+1)}](w)*1)}{
	    \infer{\cdots, \boxed{\forall x. \psi(x)} \quad (@ V[\vec{P}^{(k+1)}](v))}{
		\vdots
		}
	}
	&
	\infer*{\Gamma, \overline{\varphi(t)}}{	    
	}   
}
$$

  then there exist $r<l_{k+2-i}$ and $s<l_{k+1}$ such that:
  \begin{enumerate}
   \item $P^{(k+2-i)}_{r}=e^{i-1}(P^{(k+1)}_{s})$.
   \item $\lambda(V[\vec{P}^{(k+1)}_{\leq s}](c(T_{s})))=\Delta$ is derived by $\exists$-rule of the following form:
     \begin{prooftree}
 \AxiomC{$\Delta, {\Psi(u)}$}\RightLabel{($\exists x.\overline{\Psi(x)} \in \Delta$)}
    \UnaryInfC{$\Delta$}
 \end{prooftree}
 where $\NN \models E[\vec{P}^{(k+1)}_{\leq s}](\Psi(u))\leftrightarrow E[\vec{P}^{(k+1)}](\psi(W[\vec{P}^{(k+1)}](\forall x.\psi(x))))$.

 \item If $n-1=0$, $V[\vec{P}^{(k+1)}](w)*0 \subseteq V[\vec{P}^{(k+1)}_{\leq s}](c(T_{s}))$.
 The following figure comprehends the situation:
      $$
\infer{\Gamma \quad (@ V[\vec{P}^{(k+1)}](w))}{
   \infer*{\Gamma, \varphi(0)  \quad (@ V[\vec{P}^{(k+1)}](w)*0)}{
   \infer{\Delta \quad (@ V[\vec{P}_{\leq s}^{(k+1)}](c(T_{s})))}{
	     \infer{\Delta, \overline{\Psi(u)} \quad (@ V[\vec{P}_{\leq s}^{(k+1)}](c(T_{s}))*0)}{
	     \vdots}
	    }		
   }
   &
	\infer*{\Gamma, \boxed{\overline{\varphi(a)}},\varphi(a+1) \quad (@ V[\vec{P}^{(k+1)}](w)*1)}{
	    \infer{\cdots, \boxed{\forall x. \psi(x)} \quad (@ V[\vec{P}^{(k+1)}](v))}{
		\vdots
		}
	}
	&
	\infer*{\Gamma, \overline{\varphi(t)}}{	    
	}   
}
$$
   (where $u$ gives a candidate of counterexample of $E[\vec{P}^{(k+1)}](\forall x. \psi(x))$ under $E[\vec{P}^{(k+1)}_{\leq s}]$.)

\item Otherwise, $V[\vec{P}^{(k+1)}](w)*1 \subseteq V[\vec{P}^{(k+1)}_{\leq s}](c(T_{s}))$.
Moreover, $E[\vec{P}^{(k+1)}_{\leq s}](a)=n-2$:
     $$
\infer{\Gamma \quad (@ V[\vec{P}^{(k+1)}](w))}{
   \infer*{\Gamma, \varphi(0)}{
   }
   &
	\infer*{\Gamma, \overline{\varphi(a)},\varphi(a+1) \quad (@ V[\vec{P}^{(k+1)}](w)*1)}{
	    \infer{\Delta \quad (@ V[\vec{P}_{\leq s}^{(k+1)}](c(T_{s})))}{
	     \infer{\Delta, \overline{\Psi(u)} \quad (@ V[\vec{P}_{\leq s}^{(k+1)}](c(T_{s}))*0)}{
	     \vdots}
	    }	
	  }
	&
	\infer*{\Gamma, \overline{\varphi(t)}}{	    
	}   
}
$$
   (where $u$ gives a candidate of counterexample of $E[\vec{P}^{(k+1)}](\forall x. \psi(x))$ under $E[\vec{P}^{(k+1)}_{\leq s}]$.
   Note that $E[\vec{P}^{(k+1)}_{\leq s}](a+1)=n-1=E[\vec{P}^{(k+1)}](a)$.)
   \end{enumerate}
   
  Similarly, consider the case when $(V[\vec{P}^{(k+1)}](v), \forall x. \psi(x) )$ is an ancestor of \\
  $(V[\vec{P}^{(k+1)}](w*n),\varphi(a+1))$ and is in $p\Pi_{i}$ ($i \in [2,k]$).
      $$
\infer{\Gamma \quad (@ V[\vec{P}^{(k+1)}](w))}{
   \infer*{\Gamma, \varphi(0)}{
   }
   &
	\infer*{\Gamma, \overline{\varphi(a)}, \boxed{\varphi(a+1)} \quad (@ V[\vec{P}^{(k+1)}](w)*1)}{
	    \infer{\cdots, \boxed{\forall x. \psi(x)} \quad (@ V[\vec{P}^{(k+1)}](v))}{
		\vdots
		}
	}
	&
	\infer*{\Gamma, \overline{\varphi(t)}}{	    
	}   
}
$$

  Then there exist $r<l_{k+2-i}$ and $s<l_{k+1}$ such that:
  \begin{enumerate}
   \item $P^{(k+2-i)}_{r}=e^{i-1}(P^{(k+1)}_{s})$.
   \item $\lambda(V[\vec{P}^{(k+1)}_{\leq s}](c(T_{s})))=\Delta$ is derived by $\exists$-rule of the following form:
     \begin{prooftree}
 \AxiomC{$\Delta, \overline{\Psi(u)}$}\RightLabel{($\exists x. \overline{\Psi(x)} \in \Delta$)}
    \UnaryInfC{$\Delta$}
 \end{prooftree}
 where $\NN \models E[\vec{P}^{(k+1)}_{\leq s}](\Psi(u))\leftrightarrow E[\vec{P}^{(k+1)}](\psi(W[\vec{P}^{(k+1)}](\forall x.\psi(x))))$.

 \item If $n-1=E[\vec{P}^{(k+1)}](t)-1$, then $V[\vec{P}^{(k+1)}](w)*2 \subseteq V[\vec{P}^{(k+1)}_{\leq s}](c(T_{s}))$.
 The following figure comprehends the situation:
  $$
\infer{\Gamma \quad (@ V[\vec{P}^{(k+1)}](w))}{
   \infer*{\Gamma, \varphi(0)}{
   }
   &
	\infer*{\Gamma, \overline{\varphi(a)}, \boxed{\varphi(a+1)} \quad (@ V[\vec{P}^{(k+1)}](w)*1)}{
	\infer{\cdots, \boxed{\forall x. \psi(x)} \quad (@ V[\vec{P}^{(k+1)}](v))}{
		\vdots
	    }
	}
	&
	\infer*{\Gamma, \overline{\varphi(t)} \quad (@ V[\vec{P}^{(k+1)}](w)*2)}{
	    \infer{\Delta \quad (@ V[\vec{P}_{\leq s}^{(k+1)}](c(T_{s})))}{
	     \infer{\Delta, \overline{\Psi(u)} \quad (@ V[\vec{P}_{\leq s}^{(k+1)}](c(T_{s}))*0)}{
	     \vdots}
	    }		
	   }
}
$$
   (where $u$ gives a candidate of counterexample of $E[\vec{P}^{(k+1)}](\forall x. \psi(x))$ under $E[\vec{P}^{(k+1)}_{\leq s}]$.)

\item Otherwise, $V[\vec{P}^{(k+1)}](w)*1 \subseteq V[\vec{P}^{(k+1)}_{\leq s}](c(T^{(k+2-i)}_{s}))$.
Moreover, $E[\vec{P}^{(k+1)}_{\leq s}](a)=n$:
    $$
\infer{\Gamma \quad (@ V[\vec{P}^{(k+1)}](w))}{
   \infer*{\Gamma, \varphi(0) }{
   }
   &
	\infer*{\Gamma, \overline{\varphi(a)},\varphi(a+1) \quad (@ V[\vec{P}^{(k+1)}](w)*1)}{
	    \infer{\Delta \quad (@ V[\vec{P}_{\leq s}^{(k+1)}](c(T_{s})))}{
	     \infer{\Delta, \overline{\Psi(u)} \quad (@ V[\vec{P}_{\leq s}^{(k+1)}](c(T_{s}))*0)}{
	     \vdots}
	    }	
	  }
	&
	\infer*{\Gamma, \overline{\varphi(t)}}{	    
	}   
}
$$
   (where $u$ gives a candidate of counterexample of $E[\vec{P}^{(k+1)}](\forall x. \psi(x))$ under $E[\vec{P}^{(k+1)}_{\leq s}]$.
    Note that $E[\vec{P}^{(k+1)}](a+1)=n=E[\vec{P}^{(k+1)}_{\leq s}](a)$.)
   \end{enumerate}

 \end{enumerate}
 
 \end{enumerate}

 \end{defn}

Now, we proceed to a concrete statement and a construction.

\begin{thm}\label{ProofIsStrategyGk}
Let $k \geq 1$.
    Suppose $I\Sigma_{k+1}(X) \vdash TI(\prec)$.
Then there exists $h \in \NN$ such that \textbf{Prover} has a winning strategy for $\mathcal{G}_{k+1}(\prec,h)$.
\end{thm}

\begin{proof}
     By assumption, we have an $I\Sigma_{k+1}(X)$-derivation $\pi=(\tau,\lambda)$ of $TI(\prec)$ in free variable normal form.
 Let $h := height(\tau)$.
 We consider \textbf{Prover}'s winning strategy for $\mathcal{G}_{k+1}(\prec,h)$.
 We suppress the inputs $(\prec,h)$ in the following for readability.
 
 By relabeling the vertices of $\tau$, we may assume that $\tau$ has a canonical labeling so that:
 \[\sigma*k \in \tau \ \& \ l \leq k \Longrightarrow \sigma*l \in \tau.\]
 
 By subformula-property of cut-free sequent calculus, each $\varphi$ appearing in $\pi$ is either $s\Sigma_{k+1}(X)$, $s\Pi_{k+1}(X)$, or $X$-free.
 
  We extract a winning strategy of \textbf{Prover} from $\pi$.
  Assume that the play so far is $\vec{P}^{(k+1)}= (P^{(k+1)}_{0}, \ldots, P^{(k+1)}_{l_{k+1}})\in \mathcal{S}_{k+1}$.
  Recall the notations $\vec{P}^{(j)}=(P^{(j)}_{0}, \ldots, P^{(j)}_{l_{j}})$, $\vec{P}^{(j)}_{\leq s}$, $T_{s}$, $\rho_{s}$, $T$ and $\rho$, introduced right before Definition \ref{doubledagar}.
 Inductively on $l_{k+1}$, we define $V,E,W, f_{1},f_{2}$ enjoying the condition (\dag\dag). (Recall Lemma \ref{structureofGk} when we define $V,E,W$.)
 
 Consider the base case $l_{k+1}=0$.
 $P^{(k+1)}_{0}$ is the initial position.
 Note that $l_{j}=0$ for $j \in [1,k+1]$ in this case, and $T=T_{0}=T_{k+1}(P^{(k+1)}_{0})$ is the tree of height $0$.
 Set:
  \begin{itemize}
  \item $V[\vec{P}^{(k+1)}] \colon T_{0} \mapsto \tau; \emptyset \mapsto \emptyset$,
  \item $E[\vec{P}^{(k+1)}] := \emptyset$ since there is no free variable in $\lambda(V[\vec{P}^{(k+1)}](\emptyset))$,
  \item $W[\vec{P}^{(k+1)}]; (\forall y. \overline{y \in X}) \mapsto m_0$.
  \end{itemize}
  $f_{1}(\vec{P}^{(k+1)})$ and $f_{2}(\vec{P}^{(k+1)},Q,\rho')$ are described comprehensively in the following inductive way.
  
 Suppose we are at a play $\vec{P}^{(k+1)}=(P^{(k+1)}_{0}, \ldots, P^{(k+1)}_{l_{k+1}})$, and $V[\vec{P}^{(k+1)}]$, $E[\vec{P}^{(k+1)}]$, $W[\vec{P}^{(k+1)}]$ satisfying $(\dag\dag)$ are already defined.
 
 We describe $f_{1}(\vec{P}^{(k+1)})$, $f_{2}(\vec{P}^{(k+1)},Q,\rho')$, and $V[\vec{P}^{(k+1)}P^{(k+1)}_{\star}]$, $E[\vec{P}^{(k+1)}P^{(k+1)}_{\star}]$, $W[\vec{P}^{(k+1)}P^{(k+1)}_{\star}]$ for the next position $P^{(k+1)}_{\star}$.
 Let $v:=c(T)$.
  We split cases by the rule deriving $\lambda(V[\vec{P}^{(k+1)}](v))$ in $\pi$.
  $\lambda(V[\vec{P}^{(k+1)}](v))$ is not an Initial Sequent since, if it was the case, then the literals $L$ and $\overline{L}$ in $\lambda(V[\vec{P}^{(k+1)}](v))$ should be falsified by $(E[\vec{P}^{(k+1)}],\rho)$, which is absurd.
  Note that, in each case below, the conditions above remain satisfied:
 \begin{enumerate}
  
    \item The cases when $\lambda(V[\vec{P}^{(k+1)}](v))=\Gamma$ is derived by True Sentence, $\lor$-Rule, $\land$-Rule, or $\forall$-Rule are all dealt with similarly, modifying the corresponding argument for Theorem \ref{ProofIsStrategy1} straightforwardly.
   As for $p\Sigma_{k+1}(X)$-Cut and $p\Sigma_{k+1}(X)$-Induction, follow the $p\Sigma_{1}(X)$-counterparts in the proof of Theorem \ref{ProofIsStrategy1}.

      \item The case when  $\lambda(V[\vec{P}^{(k+1)}](v))=\Gamma$ is derived by $\exists$-Rule: 
   \begin{prooftree}
 \AxiomC{$\Gamma, \varphi(u)$}
  \RightLabel{\quad (where $\exists x. \varphi(x) \in \Gamma$)}
  \UnaryInfC{$\Gamma$}
 \end{prooftree}
 We split cases according to the form of $\exists x. \varphi(x)$.
  \begin{enumerate}
  \item If $\exists x. \varphi(x)$ is $X$-free or a $\Delta_{0}(X)$-formula,
  then move analogously to Case (\ref{alreadydetermined}) in the proof of Theorem \ref{ProofIsStrategy1}.
   \item The case when 
   \begin{align*}
       \exists x. \varphi(x) \equiv \exists x_0\forall x_1. ( x_0 \in X \land (\overline{x_1 \in X} \lor \overline{x_{1} \prec x_{0}}) 
   \end{align*}
  is analogous to Case (\ref{subformulaofTI}) in the proof of Theorem \ref{ProofIsStrategy1}.

   \item Otherwise, $(V[\vec{P}^{(k+1)}](v),\exists x. \varphi(x))$ is an ancestor of some $(V[\vec{P}^{(k+1)}](w*n), \Psi)$, where $\Psi$ is eliminated by $p\Sigma_{k+1}(X)$-Induction or $p\Sigma_{k+1}(X)$-Cut deriving $\lambda(V[\vec{P}^{(k+1)}](w))$. 
   The latter case is simpler than the former, so we focus on the former case.
   
   Suppose $\lambda(V[\vec{P}^{(k+1)}](w))=\Delta$ be derived by the following $p\Sigma_{k+1}(X)$-Induction:
 \begin{prooftree}
 \AxiomC{$\Delta, \psi(0)$}
 \AxiomC{$\Delta, \overline{\psi(a)},\psi(a+1)$}
  \AxiomC{$\Delta, \overline{\psi(t)}$}\RightLabel{\quad ($\psi \in p\Sigma_{k+1}(X)$)}
  \TrinaryInfC{$\Delta$}
 \end{prooftree}
   
 $\Psi$ is either $\psi(0)$, $\overline{\psi(a)}$, $\psi(a+1)$, or $\overline{\psi(t)}$.

   \begin{enumerate}
    \item If $\exists x. \varphi(x)$ is $p\Sigma_{k+1}(X)$, then $\Psi \equiv \exists x. \varphi(x)$, and it is $\psi(0)$ or $\psi(a+1)$.
    
    Set $f_{1}(\vec{P}^{(k+1)}):=\emptyset$.
 Let $\rho'$ be \textbf{Delayer}'s answer.
 Define
 \begin{align}\label{proceedtotheright}
 f_{2}(\vec{P}^{(k+1)},\emptyset,\rho'):= \left\langle 1, w \right\rangle.
 \end{align}
 (We will check that \textbf{Prover} does not lose by this option below.)

 As for $V,E,W$, we split the cases according to the form of $\exists x. \varphi(x)$.    
 Let $E'$ be the restriction of $E[\vec{P}^{(k+1)}]$ to $fv(\Delta)$ and $W'$ be the restriction of $W[\vec{P}^{(k+1)}]$ to formulae in $\Delta$ in the following. 
    
   We first consider the subcase $\exists x. \varphi(x) \equiv \psi(0)$.
   By $(\dag\dag)$ (\ref{doubledagarInduction}) of Induction Hypothesis, $w*0 \subseteq v$, $\lambda(V[\vec{P}^{(k+1)}](w*0))=\Delta\cup\{\psi(0)\}$, and $w*1 \in T$ by $(\dag\dag)$:
       $$
\infer{\Delta \quad (@ V[\vec{P}^{(k+1)}](w))}{
   \infer*{\Delta, \boxed{\psi(0)} \quad (@ V[\vec{P}^{(k+1)}](w*0))}{
   	\infer{\Gamma \quad (@ V[\vec{P}^{(k+1)}](v))}{
	     \infer{\Gamma,\boxed{\varphi(u)} \quad (@ V[\vec{P}^{(k+1)}](v*0)}{
	     \vdots}
	    }	
   }
   &
	\infer*{\Delta, \overline{\psi(a)},\psi(a+1)}{
   }
	&
	\infer*{\Delta, \overline{\psi(t)}}{	    
	}   
}
$$
   
   In particular, \textbf{Prover} does not lose by the option (\ref{proceedtotheright}).
   
 For the next position $P^{(k+1)}_{\star}$ with $T_{k+1}(P^{(k+1)}_{\star}) = T'$, set 
 \begin{align*}
 V[\vec{P}^{(k+1)}P^{(k+1)}_{\star}] &:=V[\vec{P}^{(k+1)}]\restriction T',\\
      E[\vec{P}^{(k+1)}P^{(k+1)}_{\star}] &:= \begin{cases}
      E' &\mbox{(if $E[\vec{P}^{(k+1)}](t)=0$)}\\
      E' \sqcup \{a \mapsto 0\} &\mbox{(otherwise)}
      \end{cases},\\
     W[\vec{P}^{(k+1)}P^{(k+1)}_{\star}] &:= 
     \begin{cases}
     W' \sqcup \{\overline{\psi(t)} \mapsto E[\vec{P}^{(k+1)}](u)\} &\mbox{(if $E[\vec{P}^{(k+1)}](t)=0$)}\\
     W' \sqcup \{\overline{\psi(a)} \mapsto E[\vec{P}^{(k+1)}](u)\} &\mbox{(otherwise)}
     \end{cases}.
 \end{align*}
Note that
\[\lambda(V[\vec{P}^{(k+1)}](w*1))=
\begin{cases}
\Delta \cup \{\overline{\psi(t)}\}&\mbox{(if $E[\vec{P}^{(k+1)}](t)=0$)}\\
\Delta \cup \{\overline{\psi(a)}, \psi(a+1)\}  &\mbox{(otherwise)}
\end{cases}.\]

Now, we consider the subcase $\exists x. \varphi(x) \equiv \psi(a+1)$.
 Let $w*n \subseteq v$.
 By $(\dag\dag)$ (\ref{doubledagarInduction}) of Induction Hypothesis, $n \in [1,E[\vec{P}^{(k+1)}](t)]$, 
 \[\lambda(V[\vec{P}^{(k+1)}](w*n))=\Delta\cup \{ \overline{\psi(a)},\psi(a+1)\},\] 
 $E[\vec{P}^{(k+1)}](a) = n-1$, and $w*(n+1) \in T$:
       $$
\infer{\Delta \quad (@ V[\vec{P}^{(k+1)}](w))}{
   \infer*{\Delta, \psi(0)}{
   		    }	
   &
	\infer*{\Delta, \overline{\psi(a)},\boxed{\psi(a+1)} \quad (@ V[\vec{P}^{(k+1)}](w*n))}{
	\infer{\Gamma \quad (@ V[\vec{P}^{(k+1)}](v))}{
	     \infer{\Gamma,\boxed{\varphi(u)} \quad (@ V[\vec{P}^{(k+1)}](v*0))}{
	     \vdots}
       } 
   }
	&
	\infer*{\Delta, \overline{\psi(t)} \quad \left(@ V[\vec{P}^{(k+1)}](w*(e+1))\right)}{	    
	}   
}
$$
(where $e:=E[\vec{P}^{(k+1)}](t)$.)
 In particular, \textbf{Prover} does not lose by choosing the option (\ref{proceedtotheright}).

 If $n=e$, then $\lambda(V[\vec{P}^{(k+1)}](w*(n+1)))=\Delta \cup \{\overline{\psi(t)}\}$.
  For the next position $P^{(k+1)}_{\star}$ with $T(P^{(k+1)}_{\star}) = T'$, set 
 \begin{align*}
 V[\vec{P}^{(k+1)}P^{(k+1)}_{\star}] &:=V[\vec{P}^{(k+1)}]\restriction T',\\
      E[\vec{P}^{(k+1)}P^{(k+1)}_{\star}] &:= E' \\
     W[\vec{P}^{(k+1)}P^{(k+1)}_{\star}] &:= W' \sqcup \{\overline{\psi(t)} \mapsto E[\vec{P}^{(k+1)}](u)\}.
 \end{align*}

 If $n < e$, then $\lambda(V[\vec{P}^{(k+1)}](w*(n+1)))=\Delta \cup \{\overline{\psi(a)}, \psi(a+1)\}$.
 For the next position $P^{(k+1)}_{\star}$ with $T(P^{(k+1)}_{\star}) = T'$, set 
 \begin{align*}
 V[\vec{P}^{(k+1)}P^{(k+1)}_{\star}] &:=V[\vec{P}^{(k+1)}]\restriction T',\\
     E[\vec{P}^{(k+1)}P^{(k+1)}_{\star}] &:= E' \sqcup \{a \mapsto n\},\\
     W[\vec{P}^{(k+1)}P^{(k+1)}_{\star}] &:= W' \sqcup \{\overline{\psi(a)} \mapsto E[\vec{P}^{(k+1)}](u)\}.
 \end{align*}
 
   \item\label{intermediatebacktrack} Now, consider the case when $\exists x. \varphi(x)$ is $p\Sigma_{i}(X)$ ($i \in [2,k]$).
   In this case, there further exists $w'$ such that:
   \begin{itemize}
   \item $w \subsetneq w' \subsetneq v$,
   \item $(V[\vec{P}^{(k+1)}](v), \exists x. \varphi(x))$ is an ancestor of $(V[\vec{P}^{(k+1)}](w'), \chi)$, where
   \begin{align*}
   \chi\equiv\forall y \exists x. \varphi(x,y) &\in \lambda(V[\vec{P}^{(k+1)}](w')),\\
   \exists x. \varphi(x,y) &\equiv \exists x. \varphi(x),
   \end{align*}
   and $\forall$-Rule is used to derive $\lambda(V[\vec{P}^{(k+1)}](w'))$, eliminating $\exists x. \varphi(x,a')$.
    Here, $a'$ is the eigenvariable.

   \item $(V[\vec{P}^{(k+1)}](w'), \chi)$ is an ancestor of $(V[\vec{P}^{(k+1)}](w), \Psi)$.
   \end{itemize}
   
   By $(\dag\dag)$ (\ref{doubledagarInduction}) of Induction Hypothesis, there exist $r<l_{k+2-i-1}$ and $s < l_{k+1}$ such that:
   \begin{itemize}
   \item $P^{(k+2-i-1)}=e^{i}(P^{(k+1)}_{s})$.
  \item $\lambda(V[\vec{P}^{(k+1)}_{\leq s}](c(T_{s})))=\Theta$ is derived by $\exists$-rule of the following form:
     \begin{prooftree}
 \AxiomC{$\Theta, \forall x . \overline{\varphi(x,u')}$}\RightLabel{($\exists y \forall x . \overline{\varphi(x,y)} \in \Theta$)}
    \UnaryInfC{$\Theta$}
 \end{prooftree}
 where $\NN \models E[\vec{P}^{(k+1)}_{\leq s}](\exists x .\varphi(x,u'))\leftrightarrow E[\vec{P}^{(k+1)}](\exists x. \varphi(x,W[\vec{P}^{(k+1)}](\forall y \exists x.\varphi(x,y))))$.
 \end{itemize}
 
 The following figure represents the situation in case when $i=k$ and $w*1 \subseteq v$ (and therefore $\Psi \equiv \overline{\psi(a)} \equiv \forall y \exists x. \varphi(x,y)$):
       $$
\infer{\Delta \quad (@ V[\vec{P}^{(k+1)}](w))}{
   \infer*{\Delta, \psi(0) \quad (@ V[\vec{P}^{(k+1)}](w*0))}{
     \infer{\Theta \quad (@ V[\vec{P}^{(k+1)}_{\leq s}](c(T_{s})))}{
      \infer{\Theta, \forall x. \overline{\varphi(x,u')}}{
       \vdots}
       }
    }	
   &
	\infer*{\Delta, \boxed{\overline{\psi(a)}},\psi(a+1) \quad (@ V[\vec{P}^{(k+1)}](w*1))}{
	\infer{\cdots, \boxed{\forall y \exists x. \varphi(x,y)} \quad (@ V[\vec{P}^{(k+1)}](w'))}{
	  \infer*{\cdots, \boxed{\forall y \exists x. \varphi(x,y)}, \boxed{\exists x. \varphi(x,a')} }{
	   \infer{\Gamma \quad (@ V[\vec{P}^{(k+1)}](v))}{
	     \infer{\Gamma,\boxed{\varphi(u)} \quad (@ V[\vec{P}^{(k+1)}](v)*0)}{
	     \vdots}
       } 
	  }
	}
	}
	&
	\infer*{\Delta, \overline{\psi(t)}}{	    
	}   
}
$$
 
 Set $f_{1}(\vec{P}^{(k+1)}):=\emptyset$.
 Let $\rho'$ be \textbf{Delayer}'s answer.
 Define
 \[f_{2}(\vec{P}^{(k+1)},f_1(\vec{P}^{(k+1)}),\rho'):= \left\langle k+2-i, r \right\rangle.\]
   
   Note that 
   \[height(c(T_{(k+2-i-1)}(P^{(k+2-i-1)}_{r}))) = height(V[\vec{P}^{(k+1)}_{\leq s}](c(T_{s}))) < h,\]
    and therefore \textbf{Prover} does not lose by this option.
    
    For the next position $P^{(k+1)}_{\star}$, 
    \begin{itemize}
     \item let $V[\vec{P}^{(k+1)}P^{(k+1)}_{\star}]$ be the extension of $V[\vec{P}^{(k+1)}]$ mapping the child of $c(T_{s})$ to the unique child of $V[\vec{P}^{(k+1)}_{\leq s}](c(T_{s}))$.
     \item define $E[\vec{P}^{(k+1)}P^{(k+1)}_{\star}]:=E[\vec{P}^{(k+1)}_{\leq s}]$.
     \item Set $W[\vec{P}^{(k+1)}P^{(k+1)}_{\star}] := W[\vec{P}^{(k+1)}_{\leq s}] \sqcup \{\forall x. \overline{\varphi(x,u')} \mapsto E[\vec{P}^{(k+1)}](u)\}$.
    \end{itemize}


   \item Lastly, consider the case when $\exists x. \varphi(x)$ is $p\Sigma_{1}(X)$. In particular, $\varphi(u)$ is $\Delta_0(X)$, and therefore there exists the minimum finite subset $Q\subseteq \NN$ such that any finite partial predicate covering $Q$ determines the truth value of $\varphi(u)$ under the assignment $E[\vec{P}^{(k+1)}]$.
  Set $f_1(\vec{P}^{(k+1)}) := Q$.
  
 Let $\rho'$ be \textbf{Delayer}'s answer, and without loss of generality, we may assume that $\rho \subseteq \rho'$. Then either:
\[(E[\vec{P}^{(k+1)}], \rho') \Vdash \varphi(u) \ \mbox{or} \ (E[\vec{P}^{(k+1)}], \rho') \Vdash \overline{\varphi(u)}.\]
In the latter case, set $f_{2}(\vec{P}^{(k+1)},f_1(\vec{P}^{(k+1)}),\rho'):= \left\langle 0, 0 \right\rangle$, and define 
\[V[\vec{P}^{(k+1)}P^{(k+1)}_{\star}],E[\vec{P}^{(k+1)}P^{(k+1)}_{\star}],W[\vec{P}^{(k+1)}P^{(k+1)}_{\star}]\]
 similarly to the case of True Sentence.
 
 In the former case, backtrack as in the previous item (\ref{intermediatebacktrack}), formally putting $i:=1$.
 \end{enumerate}
 \end{enumerate}
 \end{enumerate}
 
 This completes the description of \textbf{Prover}'s strategy. 
 Since \textbf{Prover} can continue the play as long as \textbf{Delayer} can make a move, \textbf{Prover}'s strategy above is a winning one.

\end{proof}

Together with Proposition \ref{DelayerwinsGk+1}, we obtain the following:
\begin{cor}
Let $k \geq 1$.
 If $\pi=(\tau,\lambda) \colon I\Sigma_{k+1}(X) \vdash TI(\prec)$ and $|{\prec}| > 2_{k}(\omega^{h+1}) \cdot 2 +\omega$, then $height(\pi) > h$.

   In particular, if $|{\prec}| \geq \omega_{k+2}$, then $I\Sigma_{k+1}(X) \not\vdash TI(\prec)$.

\end{cor}

\section{Acknowledgement}
The author is deeply grateful to Toshiyasu Arai for his patient and sincere support, numerous constructive suggestions, and for his leading my attention to ordinal analysis. 
We also thank Mykyta Narusevych, Jan Kraj\'{i}\v{c}ek, Neil Thapen, Erfan Khaniki, and Pavel Pudl\'{a}k for their comments.
This work is supported by JSPS KAKENHI Grant Number 22KJ1121, Grant-in-Aid for JSPS Fellows, and
FoPM program at the University of Tokyo.

\section{Appendix}\label{Appendix}
\subsection{The observation in Proposition \ref{DelayerwinsG1} is tight}
\begin{prop}\label{ProverwinsG1}
Let $h > 0$. If $|{\prec}| \leq \omega^{h+1}+\omega$, then \textbf{Prover} has a winning strategy for $\mathcal{G}_{1}(\prec,h)$.
\end{prop}

\begin{proof}
The following strategy of \textbf{Prover} suffices:
\begin{enumerate}
 \item First, \textbf{Prover} forces \textbf{Delayer} to answer $\rho(x)=1$ for some $x < \omega^{h+1}$.
 Without loss of generality, we may assume that the initial position is of the form: 
 \[(T_0, \{\omega^{h+1} + n \mapsto 1\}).\]
 Then cast the query $Q:=[\omega^{h+1},\omega^{h+1} + n]$.
 Let $\rho'$ be \textbf{Delayer}'s answer.
 Then $(\rho')^{-1}(1)$ must include an element less than $\omega^{h+1}$.
 Suppose 
 \[\min (\rho')^{-1}(1) = \omega^{h}\cdot k_{1} + \cdots + \omega^1 \cdot k_{h}+\omega^0 \cdot k_{h+1} \quad (\forall i.\ k_i< \omega).\]
 Choose the option $\langle 0,k_{i}+1 \rangle$.
 Note that the next position $(T',\rho'\cup\rho)$ satisfies $height(T')=1$.
 \item Suppose the current position is $(T,\rho)$, 
 \[\min \rho^{-1}(1) = \omega^{h}\cdot k_{1} + \cdots + \omega^1 \cdot k_{h}+\omega^0 \cdot k_{h+1} \quad (\forall i.\ k_i< \omega),\]
 $g=height(T)\in [1,h]$,
 $c(T) = (i_1, \ldots, i_{g})$, and 
 \[(i_1, \ldots, i_{l-1},i_l + j) \in T \quad (\forall l \in [1,g] \forall j \in [0,k_l+1]).\]

 Then cast the query $[\omega^{h}\cdot k_{1} + \cdots + \omega^1 \cdot k_{h}, \min \rho^{-1}(1)]$.
 If \textbf{Delayer} cannot answer, \textbf{Prover} wins.
 Hence, suppose \textbf{Delayer} answered $\rho'$ and let $m:=\min \dom(\rho')$.
 Similarly to the previous item, we have $m < \omega^{h}\cdot k_{1} + \cdots + \omega^1 \cdot k_{h}$.
 Write 
 \[m = \omega^{h}\cdot k_{1} + \cdots + \omega^{h-e+1}\cdot k_{e} + \omega^{h-e}\cdot l_{e+1}+ \cdots +\omega^0 \cdot l_{h+1} \quad (e+1 \leq h,\ l_{e+1}<k_{e+1}).\]

 Now, choose an option as follows:
 \begin{enumerate}
     \item If $e > g$, play $\left\langle 0, k_g+1\right\rangle$. 
     \item If $e \leq g$, play $\left\langle 1, \ceil{(i_1,\ldots,i_e)}\right\rangle$.
 \end{enumerate}
 Note that \textbf{Prover} does not lose by taking these options because of the assumption, and the next position again satisfies the corresponding assumption.
\end{enumerate}
In this way, \textbf{Prover} can continue playing without losing, and therefore  \textbf{Prover} wins eventually by Lemma \ref{G1determined}.
\end{proof}

\end{document}